\documentclass{article}

\usepackage{multirow}
\usepackage{amsmath,amssymb,amsthm}
\usepackage{microtype}
\usepackage{graphicx}
\usepackage[caption=false]{subfig}
\usepackage{xcolor}
\usepackage{url}
\usepackage{booktabs} 
\usepackage{soul}  
\usepackage[shortlabels]{enumitem}
\usepackage[export]{adjustbox}
\graphicspath{{Images/}}

\newtheorem{theorem}{Theorem}[section]
\newtheorem{corollary}[theorem]{Corollary}
\newtheorem{lemma}[theorem]{Lemma}
\newtheorem{assumption}{Assumption}

\newcommand{\by}{{\bf y}}

\DeclareMathOperator*{\argmin}{\mathrm{arg\,min}}
\DeclareMathOperator*{\minimize}{\mathrm{minimize}}

\usepackage{hyperref}

\usepackage[accepted]{icml2021}

\icmltitlerunning{ZO-BCD}

\begin{document}
\twocolumn[
\icmltitle{A Zeroth-Order Block Coordinate Descent Algorithm\\ for Huge-Scale Black-Box Optimization}


\begin{icmlauthorlist}
\icmlauthor{HanQin Cai}{ucla}
\icmlauthor{Yuchen Lou}{hku} 
\icmlauthor{Daniel Mckenzie}{ucla}
\icmlauthor{Wotao Yin}{alibaba}
\end{icmlauthorlist}

\icmlaffiliation{ucla}{Department of Mathematics, University of California, Los Angeles, Los Angeles, CA, USA}
\icmlaffiliation{hku}{University of Hong Kong, Hong Kong SAR, China}

\icmlaffiliation{alibaba}{DAMO Academy, AliBaba US, Bellevue, WA, USA}

\icmlcorrespondingauthor{HanQin Cai}{hqcai@math.ucla.edu}

\icmlkeywords{Zeroth-order optimization, Derivative-free optimization, Black-box optimization, Sparse gradient, Block coordinate descent,  Sparse wavelet attack, Adversarial attack}  

\vskip 0.3in
]



\printAffiliationsAndNotice{}  

\begin{abstract}
We consider the zeroth-order optimization problem in the huge-scale setting, where the dimension of the problem is so large that performing even basic vector operations on the decision variables is infeasible. In this paper, we propose a novel algorithm, coined ZO-BCD, that exhibits favorable overall query complexity {\em and} has a much smaller per-iteration computational complexity. In addition, we discuss how the memory footprint of ZO-BCD can be reduced even further by the clever use of circulant measurement matrices. As an application of our new method, we propose the idea of crafting adversarial attacks on neural network based classifiers in a {\em wavelet domain}, which can result in problem dimensions of over one million. In particular, we show that crafting adversarial examples to audio classifiers in a wavelet domain can achieve the state-of-the-art attack success rate of $97.9\%$ with significantly less distortion. 
\end{abstract}

\section{Introduction}
\label{sec:Intro}
We are interested in problem \eqref{eq:Minimize_me} under the restrictive assumption that one only has noisy zeroth-order access to $f$ ({\em i.e.} one cannot access the gradient, $\nabla f)$) {\em and} the dimension of the problem, $d$, is huge, say $d>10^7$. 
\begin{equation} \label{eq:Minimize_me}
    \minimize_{x\in \mathcal{X}\subset\mathbb{R}^{d}} f(x) .
\end{equation}
Such problems (with small or large $d$) arise frequently in domains as diverse as simulation-based optimization in chemistry and physics \citep{reeja2012microwave}, hyperparameter tuning for combinatorial optimization solvers \citep{hutter2014efficient} and for neural networks \citep{bergstra2012random} and online marketing \citep{flaxman2005online}. Lately, algorithms for zeroth-order optimization have drawn increasing attention due to their use in finding good policies in reinforcement learning \citep{salimans2017evolution,mania2018simple,choromanski2020provably} and in crafting adversarial examples given only black-box access to neural-network based classifiers \citep{chen2017zoo,zo-scd,alzantot2018did,cai2020zeroth}. We note that in all of these applications queries ({\em i.e.} evaluating $f$ at a chosen point) are considered expensive, thus it is desirable for zeroth-order optimization algorithms to be as {\em query efficient} as possible. 

Unfortunately, it is known \citep{jamieson2012query} that the worst case query complexity of {\em any} noisy zeroth order algorithm for strongly convex $f$ scales linearly with $d$. Clearly, this is prohibitive for huge $d$. Recent works have begun to side-step this issue by assuming $f$ has additional, low-dimensional, structure. For example, \citep{wang2018stochastic,balasubramanian2018zeroth,cai2020scobo,cai2020zeroth} assume the gradients $\nabla f$ are (approximately) $s$-sparse (see Assumption~\ref{assumption:Sparsity}) while \citep{golovin2019gradientless} and others assume $f(x) = g(Az)$ where $A:\mathbb{R}^{s}\to \mathbb{R}^{d}$ and $s \ll d$. All of these works promise a query complexity that scales linearly with the intrinsic dimension, $s$, and only logarithmically with the extrinsic dimension, $d$. However there is no free lunch here; the improved complexity of \citep{golovin2019gradientless} requires access to {\em noiseless} function evaluations, the results of \citep{balasubramanian2018zeroth} only hold if the support of $\nabla f(x)$ is {\em fixed}\footnote{See Appendix A of \citep{cai2020zeroth} for a proof of this.} for all $x \in \mathcal{X}$ and while \citep{wang2018stochastic,cai2020scobo,cai2020zeroth} allow for noisy function evaluations and changing gradient support, both solve a computationally intensive optimization problem as a sub-routine, requiring at least $\Omega(sd\log(d))$ memory and FLOPS per iteration. 



\subsection{Contributions}
In this paper we provide the first zeroth-order optimization algorithm enjoying a sub-linear (in $d$) query complexity {\em and} a sub-linear per-iteration computational complexity. In addition, our algorithm has an exceptionally small memory footprint. Furthermore, it does not require the repeated sampling of $d$-dimensional random vectors, a hallmark of many zeroth-order optimization algorithms. With this new algorithm, ZO-BCD, in hand we are able to solve black-box optimization problems of a size hitherto unimagined. Specifically, we consider the problem of generating adversarial examples to fool neural-network-based classifiers, given only black-box access to the model (as introduced in \citep{chen2017zoo}). However, we consider generating these malicious examples by perturbing natural examples {\em in a wavelet domain}. For image classifiers (we consider \texttt{Inception-v3} trained on \texttt{ImageNet}) we are able to produce attacked images with a record low $\ell_2$ distortion of $13.7$ and a success rate of $96\%$, exceeding the state of the art. 
For audio classifiers, switching to a wavelet domain results in a problem dimension of over $1.7$ million. Using ZO-BCD, this is not an issue and we achieve a targeted attack success rate of $97.93\%$ with a mean distortion of $-6.32$ dB.
 
\begin{figure}
\label{fig:SampleAttack}
  \centering\small
  \includegraphics[width=2.2cm, valign=c]{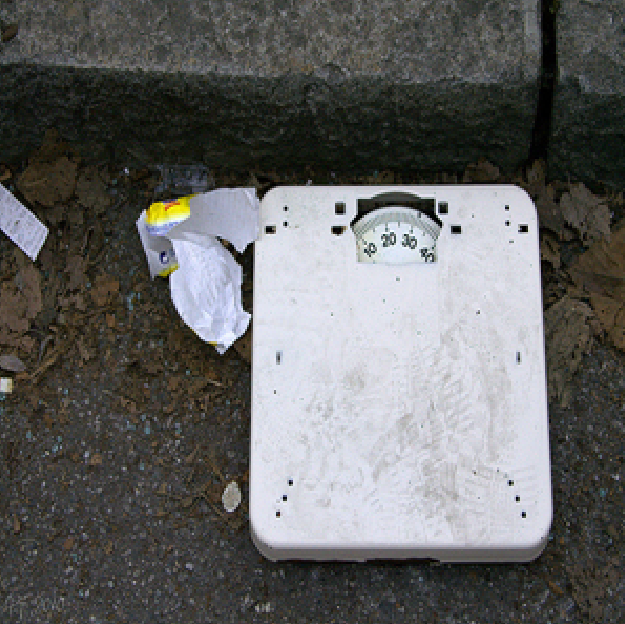}
  $+$
  $0.02\times$\includegraphics[width=2.2cm, valign=c]{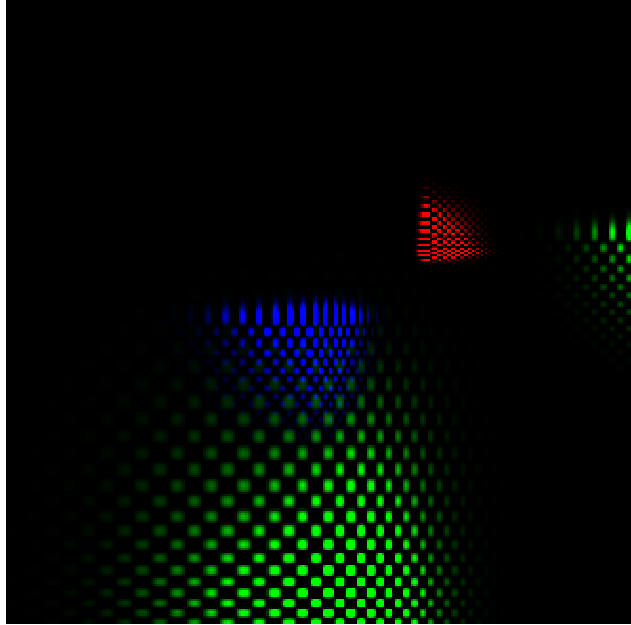}
  $=$
  \includegraphics[width=2.2cm, valign=c]{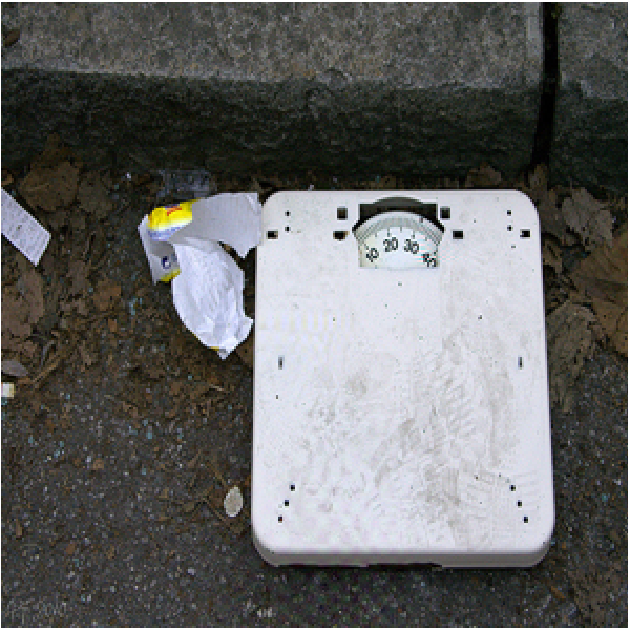}
  \caption{{\bf Left:} Original image from {\tt ImageNet} with true label ``scale''. {\bf Center:} Wavelet perturbation crafted using ZO-BCD. {\bf Right:} The attacked image, constructed by adding the perturbation, scaled down by 0.02, to the original image. Mis-classified as a ``switch''.}%
\end{figure}

\subsection{Relation to prior work}
As mentioned above, the recent works \citep{wang2018stochastic,balasubramanian2018zeroth,cai2020zeroth} provide zeroth-order algorithms whose query complexity scales linearly with $s$ and logarithmically with $d$. In order to ameliorate the prohibitive computational and memory cost associated with huge $d$, several domain-specific heuristics have been employed in the literature. For example in \citep{chen2017zoo,alzantot2019genattack}, in relation to adversarial attacks, an upsampling operator $D: \mathbb{R}^{p}\to \mathbb{R}^{d}$ with $p \ll d$ is employed. Problem \eqref{eq:Minimize_me} is then replaced with the lower dimensional problem: $\minimize_{z\in\mathbb{R}^{p}}f(D(z))$. Several other works \citep{alzantot2018did, taori2019targeted,cai2020zeroth} choose a low dimensional random subspace $T_{k}\subset\mathbb{R}^{d}$ at each iteration and then restrict $x_{k+1} - x_{k}\in T_{k}$. We emphasize that none of the aforementioned works {\em prove} such a procedure will converge, and our work is partly motivated by the desire to provide this empirically successful trick with firm guarantees of success.

In the reinforcement learning literature it is common to evaluate the $f(x_{k} + z_{k,i})$ on parallel devices and send the computed function value and the perturbation $z_{k,i}$ to a central worker, which then computes $x_{k+1}$. As $x\in\mathbb{R}^{d}$ parametrizes a neural network, $d$ can be extremely large, and hence the communication of the $z_{k,i}$ between workers becomes a bottle neck. \citep{salimans2017evolution} overcomes this with a ``seed sharing'' trick, but again this heuristic lacks rigorous analysis. We hope ZO-BCD's (particularly the ZO-BCD-RC variant, see Section~\ref{sec:zo-bcd}) intrinsically small memory footprint will make it a competitive, principled alternative. 

Finally, although two recent works have examined the idea of wavelet domain adversarial attacks \citep{wavetransform,din2020steganographic} they are of a very different nature to our approach. We discuss this further in Section~\ref{sec:sparse_attack}. 

\subsection{Assumptions and notation}
\label{sec:assumptions}
As mentioned, we will suppose the decision variables $x$ have been subdivided into $J$ blocks of sizes $d^{(1)},\ldots, d^{(J)}$. Following the notation of \citep{tappenden2016inexact}, we suppose there exists a permutation matrix $U\in\mathbb{R}^{d\times d}$ and a division of $U$ into submatrices $U = [U^{(1)},U^{(2)},\ldots, U^{(J)}]$ such that $U^{(j)}\in\mathbb{R}^{d\times d^{(j)}}$ and the $j$-th block is spanned by the columns of $U^{(j)}$. Letting $x^{(j)}$ denote the decision variables in the $j$-th block, we write $x = \sum_{j=1}^{J} U^{(j)}x^{(j)}$ or simply $x = (x^{(1)},x^{(2)},\ldots, x^{(J)})$. We shall consistently use the notation $g(x) :=\nabla f(x)$, omitting $x$ if it is clear from context. By $g^{(j)}$ we mean the components of the gradient corresponding to the $j$-th block, {\em i.e.} $g^{(j)} = \nabla_{x^{(j)}}f$, regarded as either a vector in $\mathbb{R}^{d}$ or in $\mathbb{R}^{d^{(j)}}$. Finally, we use $\tilde{\mathcal{O}}(\cdot)$ notation to suppress logarithmic factors. Let us now introduce some standard assumptions on the objective function.

\begin{assumption}[Block Lipschitz differentiability] \label{assumption:Lipschitz_Diff}
$f$ is continuously differentiable and for some fixed constant $L_{j}$
\begin{equation*}
\|g^{(j)}(x) - g^{(j)}(x + U^{(j)}t)\|_2 \leq L_{j}\|t\|_2
\end{equation*}
for all $j=1,\ldots, J$, $x\in \mathcal{X}$ and $t\in \mathbb{R}^{d^{(j)}}$. 
\end{assumption}
If $f$ is $L$-Lipschitz differentiable then it is also block Lipschitz differentiable, with $\max_j L_j \leq L$.

\begin{assumption}[Convexity]
\label{assumption:Convexity}
$\mathcal{X}$ is a convex set, and $f(tx+(1-t)y)\leq tf(x)+(1-t)f(y)$ for all $x,y\in\mathcal{X},\,\,t\in[0,1]$.
\end{assumption}
Define the solution set $ \mathcal{X}^{*} = \argmin_{x\in\mathcal{X}}f(x)$. If this set is non-empty we define the level set radius for $x\in\mathcal{X}$ as:
\begin{align}
    \mathcal{R}(x):=\max_{y\in\mathcal{X}}\max_{x^*\in \mathcal{X}^*}\{\|y-x^*\|_2:f(y)\leq f(x)\}.
    \label{eq:level_set_radius}
\end{align}
\begin{assumption}[Non-empty solution set and Bounded level sets]
\label{assumption:BoundedLevelSets}
$\mathcal{X}^{*}$ is non-empty and $\mathcal{R}(x_0) < \infty$.
\end{assumption}
\begin{assumption}[Adversarially noisy oracle] \label{assumption:noise model} $f$ is only accessible through a {\em noisy oracle}: $E_{f}(x) = f(x) + \xi$, where $\xi$ is a random variable satisfying $|\xi| \leq \sigma$.
\end{assumption}

\begin{assumption}[Sparse gradients]
\label{assumption:Sparsity}
There exists a fixed integer $0 < s_{\mathrm{exact}} < d$ such that for all $x\in \mathbb{R}^{d}$:
\begin{equation*}
    \|g(x)\|_{0} := \left|\{i: g_{i}(x) \neq 0\}\right| \leq s_{\mathrm{exact}} .
\end{equation*}
\end{assumption}
It is of interest to relax this assumption to an ``approximately sparse'' assumption, such as in \citep{cai2020zeroth}. However,  it is unclear randomly chosen blocks (see Section~\ref{sec:equisparse}) will inherit this property. We leave the analysis of this case for future work. Finally, let $\nabla^{2}_{jj}f\in\mathbb{R}^{d^{(j)}\times d^{(j)}}$ denote the $j$-th block Hessian. 

\begin{assumption}[Weakly sparse block Hessian] \label{assumption:WeakSparsity}
$f$ is twice differentiable and, for all $j = 1,\ldots, J$, $x\in \mathcal{X}$ we have $\|\nabla^{2}_{jj}f(x)\|_{1} \leq H$ for some fixed constant $H$. 
\end{assumption}
Note that $\|\cdot\|_{1}$ represents the {\em element-wise} $\ell_1$-norm: $\|B\|_1 = \sum_{i,j}|B_{ij}|$. 

\section{Gradient estimators}
\label{sec:grad estimator}
Randomized (block) coordinate descent methods are an attractive alternative to full gradient methods for huge-scale problems~\citep{nesterov2012efficiency}. 
ZO-BCD is a block coordinate method adapted to the zeroth-order setting and conceptually has three steps: 
\begin{enumerate}
\vspace{-0.05in}
    \item Choose a block, $j\in \{1,\ldots, J\}$ at random.
    \item Use zeroth-order queries to find an approximation $\hat{g}^{(j)}_{k}$ of the true block gradient $g^{(j)}_{k}$.
    \item Take a negative gradient step: $x_{k+1} = x_{k} - \alpha\hat{g}^{(j)}_{k}$.
\end{enumerate}
We abuse notation slightly; the block gradient $\hat{g}_k^{(j)}$ is regarded as both a vector in $\mathbb{R}^{d^{(j)}}$ and a vector in $\mathbb{R}^{d}$ with non-zeros in the $j$-th block only. 



In principle any scheme for constructing an estimator of $g_k$ could be adapted for estimating $g^{(j)}_{k}$, as long as one is able to bound $\|g^{(j)}_k - \hat{g}^{(j)}_{k}\|_2$ with high probability. As we wish to exploit gradient sparsity, we choose to adapt the estimator presented in \citep{cai2020zeroth}. Let us now discuss how to do so. Fix a sampling radius $\delta > 0$. Suppose the $j$-th block has been selected and choose $m$ {\em sample directions} $z_{1},\ldots, z_{m}\in\mathbb{R}^{d^{(j)}}$ from a Rademacher distribution\footnote{That is, the entries of $z_{i}$ are $+1$ or $-1$ with equal probability.}. Consider the finite difference approximations to the directional derivatives:
\begin{equation} \label{eq:Measurements}
y_{i} = \frac{1}{\sqrt{m}}\frac{E_{f}(x+\delta U^{(j)}z_{i}) - E_{f}(x)}{\delta} \approx \frac{1}{\sqrt{m}}z_{i}^{\top}g^{(j)}
\end{equation}
Stack the $y_i$ into a vector $y \in \mathbb{R}^{m}$, let $Z\in\mathbb{R}^{m\times (d/J)}$ be the matrix with rows $z_i^{\top}/\sqrt{m}$ and observe $y \approx Zg^{(j)}$; an underdetermined linear system. If\footnote{Throughout, we assume $s \geq s_{\mathrm{exact}}$ is specified by the user.} $\|g\|_0 \leq s$ (see Assumption~\ref{assumption:Sparsity}) then also $\|g^{(j)}\|_0 \leq s$. Thus, we approximate $g^{(j)}$ by solving the sparse recovery problem:
\begin{equation} \label{eq:SparseRecoveryForGradient}
    \hat{g}^{(j)} = \argmin \|Zv - y\|_2 \quad v\in\mathbb{R}^{d^{(j)}} \textnormal{ and } \|v\|_{0} \leq s .
\end{equation}
We propose solving Problem~\eqref{eq:SparseRecoveryForGradient} using $n$ iterations of CoSaMP \citep{needell2009cosamp}, but other choices are possible. This approach, presented as Algorithm~\ref{algo:grad estimate}, yields an accurate gradient estimator \citep{cai2020zeroth} using only $\tilde{\mathcal{O}}(s)$ queries, assuming $g^{(j)}$ is sparse. In contrast, direct finite differencing \citep{berahas2019theoretical} requires $\mathcal{O}(d^{(j)})$ queries. 
\begin{algorithm}[tb]
\caption{Block Gradient Estimation}
\label{algo:grad estimate}
\begin{algorithmic}[1]
   \STATE {\bfseries Input:} $x$: current point; $j$: choice of block; $s$: gradient sparsity level; $\delta$: query radius; $n$: number of CoSaMP iterations; $\{z_i\}_{i=1}^m$: sample directions in $\mathbb{R}^{d^{(j)}}$.
   \FOR{$i=1$ {\bfseries to} $m$}
   \STATE $y_i\leftarrow (E_f(x+\delta U^{(j)}z_i)-E_f(x))/(\sqrt{m}\delta)$
   \ENDFOR
   \STATE $\by\leftarrow[y_1,\cdots, y_m]^T$; $Z\leftarrow 1/\sqrt{m}[z_1,\cdots,z_m]^T$
   \STATE $\hat{g}^{(j)}\approx\arg\min_{\|v\|_0\leq s}\|Zv-\by\|_2$  by $n$ iterations of CoSaMP
   \STATE {\bfseries Output:} $\hat{g}^{(j)}$: estimated block gradient.
\end{algorithmic}
\end{algorithm}

\begin{theorem}  \label{thm:Grad_Estimate_Error}
    Suppose $f$ satisfies Assumptions~\ref{assumption:Lipschitz_Diff}, \ref{assumption:Sparsity} and \ref{assumption:WeakSparsity}. Let $g^{(j)}$ be the output of Algorithm~\ref{algo:grad estimate} with $\delta = 2\sqrt{\sigma/H}$, $s\geq s_{\mathrm{exact}}$ and $m = b_1s\log(d/J)$ Rademacher sample directions. Then with probability at least $1 - (s/d)^{b_2 s}$:
    \begin{equation}        \label{eq:ErrorBound2}
       \|\hat{g}^{(j)} - g^{(j)}\|_{2} \leq \rho^{n}\|g^{(j)}\|_2 +2\tau\sqrt{\sigma H}.
    \end{equation}
 \end{theorem}

The constants $b_1$ and $b_2$ are directly proportional; more sample directions results in a higher probability of recovery. In our experiments we consider $1\leq b_1 \leq 4$. The constant $\rho$ and $\tau$ arise from the analysis of CoSaMP. Both are inversely proportional to $b_1$. For our range of $b_1$, $\rho \approx 0.5$ and $\tau \approx 10$.
 
\subsection{Almost equisparse blocks using randomization}
\label{sec:equisparse}
Suppose $f$ satisfies Assumption~\ref{assumption:Sparsity}, so $\|g(x)\|_{0} \leq s_{\mathrm{exact}}$ for all $x$. In general one cannot improve upon the bound $\|g^{(j)}(x)\|_{0} \leq s_{\mathrm{exact}}$; perhaps all non-zero entries of $g$ lie in the $j$-th block. However, by {\em randomizing} the blocks one can guarantee, with high probability, the non-zero entries of $g$ are {\em almost equally distributed} over the $J$ blocks. We assume, for simplicity, equal-sized blocks ({\em i.e.} $d^{(j)} = d/J$). 

\begin{theorem}
\label{thm:equalsparsity}
Choose $U$ uniformly at random. For any $\Delta > 0$, $x \in \mathbb{R}^d$ we have $\|g^{(j)}(x)\|_{0} \leq (1+\Delta)s_{\mathrm{exact}}/J$ for all $j$ with probability at least $1- 2J\exp(\frac{-\Delta^2s_{\mathrm{exact}}}{3J})$.
\end{theorem}
It will be convenient to fix a value of $\Delta$, say $\Delta = 0.1$. An immediate consequence of Theorem~\ref{thm:equalsparsity} is that one can improve upon the query complexity of Theorem~\ref{thm:Grad_Estimate_Error}:
\begin{corollary}
\label{cor:block error bound}
Choose $U$ uniformly at random. For fixed $x\in\mathbb{R}^d$ the error bound \eqref{eq:ErrorBound2} in Theorem~\ref{thm:Grad_Estimate_Error} still holds, now with probability $1 - \mathcal{O}\left(J\exp(\frac{-0.01s_{\mathrm{exact}}}{3J})\right)$, for $s :=s_{\mathrm{block}} \geq 1.1s_{\mathrm{exact}}/J$ (and all other parameters the same).
\end{corollary}

This allows us to use approximately $J$ times fewer queries per iteration.

\subsection{Further reducing the required randomness}
\label{sec:Cyclic}
As discussed in \citep{cai2020zeroth}, one favorable feature of using a compressed sensing based gradient estimator is the error bound \eqref{eq:ErrorBound2} is {\em universal}. That is, it holds for all $x\in \mathbb{R}^{d}$ for the {\em same set of sample directions} $\{z_{i}\}_{i=1}^{m} \subset \mathbb{R}^{d/J}$. So, instead of resampling new vectors at each iteration we may use {\em the same} sampling directions {\em for each
block and each iteration}. Thus, only $md/J = \tilde{\mathcal{O}}(s_{\mathrm{exact}}d/J^{2})$ binary random variables need to be sampled, stored and transmitted in ZO-BCD. Remarkably, one can do even better by choosing as sample directions a subset of the rows of a circulant matrix. Recall a circulant matrix of size $d/J\times d/J$, generated by $v\in\mathbb{R}^{d/J}$, has the following form:
\begin{align}
\label{CirculantMatrix}
\mathcal{C}(v)=\begin{pmatrix} v_{1} & v_2 & \cdots & v_{d/J} \\ v_{d/J} & v_{1} & \cdots & v_{d/J-1} \\ \vdots & \ddots & \ddots & \vdots \\ v_2 & \cdots & v_{d/J} & v_{1} \end{pmatrix} .
\end{align}
Equivalently, $\mathcal{C}(v)$ is the matrix with rows $\mathcal{C}_{i}(v)$ where:
$$
\mathcal{C}_{i}(v)\in\mathbb{R}^{d/J} \quad \textnormal{ and } \quad \mathcal{C}_{i}(v)_{j} = v_{i+j-1} .
$$

By exploiting recent results in signal processing, we show:

\begin{theorem}
\label{cor:Circulant}
Assign blocks randomly as in Corollary~\ref{cor:block error bound}. Let $z\in\mathbb{R}^{d/J}$ be a Rademacher random vector. Fix $s :=s_{\mathrm{block}} \geq 1.1s_{\mathrm{exact}}/J$. Choose a random subset $\Omega = \{j_1,\ldots, j_m\} \subset \{1,\ldots, d/J\}$ of cardinality $m = b_3s\log^2(s)\log^{2}(d/J)$ and let $z_{i} = \mathcal{C}_{j_i}(z)$ for $i= 1,\ldots, m$. Let $\delta = 2\sqrt{\sigma/H}$. For fixed $x \in \mathbb{R}^d$ the error bound \eqref{eq:ErrorBound2} in Theorem~\ref{thm:Grad_Estimate_Error} still holds, now with probability at least
\begin{equation*}
1 - 2J\exp\left(\frac{-0.01s_{\mathrm{exact}}}{3J}\right) - \left(d/J\right)^{\log(d/J)\log^2(4s)}.
\end{equation*}
\end{theorem}
Hence, one only needs $d/J$ binary random variables (to construct $z$) and $m = \tilde{\mathcal{O}}\left(s_{\mathrm{exact}}/J\right)$ randomly selected integers for the {\em entire} algorithm. Note (partial) circulant matrices allow for a {\em fast multiplication}, further reducing the computational complexity of Algorithm~\ref{algo:grad estimate}.  

\section{The proposed algorithm: ZO-BCD}
\label{sec:zo-bcd}
Let us now introduce our new algorithm. We consider two variants, distinguished by the kind of sampling directions used. ZO-BCD-R uses {\bf R}ademacher sampling directions. ZO-BCD-RC uses {\bf R}ademacher-{\bf C}irculant sampling directions, as described in Section~\ref{sec:Cyclic}. For simplicity, we present our algorithm for randomly selected, equally sized coordinate blocks. With minor modifications our results still hold for user-defined and/or unequally sized blocks (see Appendix~\ref{sec:UnequalBlocks}). The following theorem guarantees both variants converge at a sublinear rate to within a certain error tolerance. As the choice of block in each iteration is random, our results are necessarily probabilistic. We say $x_{K}$ is an $\varepsilon$-optimal solution if $f(x_{K}) - f^{*} \leq \varepsilon$.

\begin{algorithm}[tb]
\caption{ZO-BCD} \label{algo:zo-bcd}
\begin{algorithmic}[1]
   \STATE {\bfseries Input:} $x_0$: initial point; $s$: gradient sparsity level; $\alpha$: step size; $\delta$: query radius; $J$: number of blocks. \\
   \STATE $s_{\textnormal{block}} \gets 1.1s/J$
   \STATE Randomly divide $x$ into $J$ equally sized blocks. \\
   \IF{ZO-BCD-R}
        \STATE $m\leftarrow b_1s_{\mathrm{block}}\log(d/J)$
        \STATE Generate Rademacher random vectors $z_{1},\ldots, z_{m}$
   \ELSIF{ZO-BCD-RC}
        \STATE $m \leftarrow b_3s_{\mathrm{block}}\log^2(s_{\mathrm{block}})\log^{2}(d/J)$
        \STATE Generate Rademacher random vector $z$.
        \STATE Randomly choose $\Omega\subset \{1,\ldots, d/J\}$ with $|\Omega| = m$
        \STATE Let $z_{i} = \mathcal{C}_{j_i}(z)$ for $i=1,\ldots m$ and $j_i\in\Omega$
    \ENDIF
   \FOR{$k=0$ {\bfseries to} $K$}
   \STATE Select a block $j\in \{1,\ldots,J\}$ uniformly at random.
   \STATE $\hat{g}_k^{(j)}\leftarrow$Gradient Estimation$(x_k^{(j)},s_{\mathrm{block}},\delta,\{z_i\}_{i=1}^{m})$
   \STATE $x_{k+1}\leftarrow x_k-\alpha\hat{g}_k^{(j)}$
   \ENDFOR
   \STATE {\bfseries Output:} $x_K$: estimated optimum point.
\end{algorithmic}
\end{algorithm}

\begin{theorem}
\label{thmNonReg}
Assume $f$ satisfies Assumptions~\ref{assumption:Lipschitz_Diff}--\ref{assumption:WeakSparsity}. Define: 
\begin{align*}
  &L_{\max}=\max_{j}L_j \quad\text{and}\quad c_1=2JL_{\max}\mathcal{R}^2(x_0).
\end{align*}
Assume $4\rho^{4n}+\frac{16\tau^2\sigma H}{c_1L_{\max}}<1$. Choose sparsity $s \geq s_{\text{exact}}$, step size $\alpha=\frac{1}{L_{\max}}$ and query radius $\delta = 2\sqrt{\sigma/H}$. Choose the number of CoSaMP iterations $n$ and error tolerance $\varepsilon$ such that:
\begin{equation*}
\frac{c_1}{2}\left(2\rho^{2n}+\sqrt{4\rho^{4n}+\frac{16\tau^2\sigma H}{c_1\zeta L_{\max}}}\right)<\varepsilon <f(x_0)-f^* .
\end{equation*}
With probability at least $1 - \zeta - \tilde{\mathcal{O}}\left( \frac{J^2}{\varepsilon}\exp\left(\frac{-0.01s_{\mathrm{exact}}}{3J}\right)\right)$ ZO-BCD-R finds an $\varepsilon-$optimal solution in $\tilde{\mathcal{O}}\left(s/\varepsilon\right)$ queries,  requiring $\tilde{\mathcal{O}}(sd/J^2)$ FLOPS per iteration and $\tilde{\mathcal{O}}(sd/J^{2})$ total memory. With probability at least
\begin{equation*}
    1 - \tilde{\mathcal{O}}\left(\frac{J^2}{\varepsilon}\exp\left(-\frac{0.01s_{\mathrm{exact}}}{3J}\right)\right) - \left(d/J\right)^{\log(d/J)\log^2(4.4s/J)}
\end{equation*}
ZO-BCD-RC finds an $\varepsilon-$optimal solution using $\tilde{\mathcal{O}}\left(s/\varepsilon\right)$ queries, $\tilde{\mathcal{O}}(d/J)$ FLOPS per iteration and $\mathcal{O}(d/J)$ total memory.
\end{theorem}

Thus, up to logarithmic factors, ZO-BCD achieves the same query complexity as ZORO \citep{cai2020zeroth} with a much lower per-iteration computational cost. We pay for the improved computational and memory complexity of ZO-BCD-RC with a slightly worse theoretical query complexity (by a logarithmic factor) due to the requirements of Theorem~\ref{cor:Circulant}. First order block coordinate descent methods typically have a probability of success $1 -\zeta$, thus in switching to zeroth-order this probability decreases by a factor which is negligible for truly huge problems ({\em e.g.} $d \approx 10^{6}$ and $s_{\mathrm{exact}} \approx 10^{4}$). For smaller problems we find randomly re-assigning the decision variables to blocks every $J$ iterations is a good way to increase ZO-BCD's probability of success.

\section{Sparse wavelet transform attacks} 
\label{sec:sparse_attack}
Adversarial attacks on neural network based classifiers is a popular application and benchmark for zeroth-order optimizers \citep{chen2017zoo,chen2019zo,ilyas2018black,modas2019sparsefool}. Specifically, let $F(x) \in [0,1]^{C}$ denote the predicted probabilities returned by the model for input signal $x$. Then the goal is to find a small distortion $\delta$ such that that the model's top-1 prediction on $x+\delta$ is no longer correct: $\mathrm{argmax}_{c=1,\ldots, C}F_{c}(x+\delta) \neq \mathrm{argmax}_{c=1,\ldots, C}F_{c}(x)$. Because we only have access to the logits, $F(x)$, not the internal workings of the model, we are unable to compute $\nabla F(x)$ and hence this problem is of zeroth order. Recently, \citep{cai2020zeroth} showed it is reasonable to assume the attack loss function exhibits (approximate) gradient sparsity, and proposed generating adversarial examples by adding a distortion to the victim image that is sparse in the image pixel domain. We extend this and propose a novel {\em sparse wavelet transform attack}, which searches for an adversarial distortion $\delta^\star$ in the wavelet domain:
\begin{equation} \label{eq:attack_object_function}
    \delta^\star = \argmin_\delta f(\mathrm{IWT}(\mathrm{WT}(x)+\delta))+\lambda\|\delta\|_0,
\end{equation}
where $x$ is a given image/audio signal, $f$ is the Carlini-Wagner loss function \citep{chen2017zoo}, $\mathrm{WT}$ is the chosen (discrete or continuous) wavelet transform, and $\mathrm{IWT}$ is the corresponding inverse wavelet transform. As
wavelet transforms extract the important features of the data, we expect the gradients of this new loss function to be even sparser than those of the corresponding pixel-domain loss function \citep[Figure~1]{cai2020zeroth}. Moreover, the inverse wavelet transform spreads the energy of the sparse perturbation, resulting in more natural-seeming attacked signals, as compared with pixel-domain sparse attacks \citep[Figure~6]{cai2020zeroth}.

\begin{enumerate}
\vspace{-0.05in}
    \item \textbf{Sparse DWT attacks.} The discrete wavelet transform (DWT) is a well-known method for data compression and denoising \citep{mallat1999wavelet,cai2012image}. Many real-world media data are compressed and stored in the form of DWT coefficients (\textit{e.g.} JPEG-2000 for images and Dirac for videos), thus attacking the wavelet domain is more direct in these cases. Since DWT does not increase the problem dimension\footnote{When using periodic boundary extension. If another boundary extension is used, the dimension of the wavelet coefficients may increase slightly, depending on the size of the filters and the level of the transform.}, the query complexity of sparse wavelet-domain attacks is the same as sparse pixel-domain attacks. An interesting variation is to only attack the important ({\em i.e.} large) wavelet coefficients. We explore this further in Section~\ref{sec:experiments}. This can reduce the attack problem dimension by $60\%$--$80\%$ for typical image datasets. Nevertheless, for large, modern color images, this dimension can still be massive.
    
    \item \textbf{Sparse CWT attacks.} For oscillatory signals, the continuous wavelet transform (CWT) with analytic wavelets is preferred \citep{mallat1999wavelet,lilly2010analytic}. Unlike DWT, the dimension of the CWT coefficients is much larger than the original signal dimension. For example, attacking even $1$ second audio clips in a CWT domain results in a problem of size $d > 1.7$ million (see Section~\ref{sec:AudioAttack})!
\end{enumerate}


The idea of adversarial attacks on DWT coefficients was also proposed in \citep{wavetransform}, but they assume a white-box model and study only dense attacks on discrete Haar wavelets. \citep{din2020steganographic} considers a ``steganographic'' attack, where the important wavelet coefficents of a target image are ``hidden'' within the wavelet transform coefficients of a victim image. We appear to be the first to connect (both discrete and continuous) wavelet transforms to sparse zeroth-order adversarial attacks.

\section{Empirical results} \label{sec:experiments}

In this section, we first show the empirical advantages of ZO-BCD with synthetic examples. Then, we demonstrate the performance of ZO-BCD in  two real-world applications: (i) sparse DWT attacks on images, and (ii) sparse CWT attacks on audio signals. We compare the two versions of ZO-BCD (see Algorithm~\ref{algo:zo-bcd}) against two venerable zeroth-order algorithms---FDSA \citep{Kiefer1952} and SPSA\footnote{SPSA using Rademacher sample directions coincides with Random Search \citep{nesterov2017random}.} \citep{spall1998overview}---as well as three more recent contributions: ZO-SCD \citep{chen2017zoo}, ZORO \citep{cai2020zeroth} and LM-MA-ES \citep{loshchilov2018large}. ZO-SCD is a zeroth-order (non-block) coordinate descent method. ZORO uses a similar gradient estimator as ZO-BCD, but computes the full gradient. LM-MA-ES is a recently proposed extension of CMA-ES \citep{hansen2001completely} to the large-scale setting. In Section~\ref{sec:ImageAttack}, we consider ZO-SGD \citep{zo-sgd}, a variance-reduced version of SPSA, as this has empirically shown better performance on this task than SPSA\footnote{Variance reduction did little to improve the performance of SPSA in the experiments of Section~\ref{sec:synthetic_experiments}.}. We also consider ZO-AdaMM \citep{chen2019zo}, a zeroth-order method incorporating momentum. The experiments in Section~\ref{sec:synthetic_experiments} were executed from Matlab 2020b on a laptop with Intel i7-8750H CPU and 32GB RAM. The experiments in Sections~\ref{sec:ImageAttack} and \ref{sec:AudioAttack} were executed on a workstation with Intel i9-9940X CPU, 128GB RAM, and two of Nvidia RTX-3080 GPUs. All code is available online at \url{https://github.com/YuchenLou/ZO-BCD}.

\subsection{Synthetic examples}
\label{sec:synthetic_experiments}

\begin{figure}[t]
\centering
\subfloat[Sparse Quadric.]{
\includegraphics[width = 0.95\linewidth]{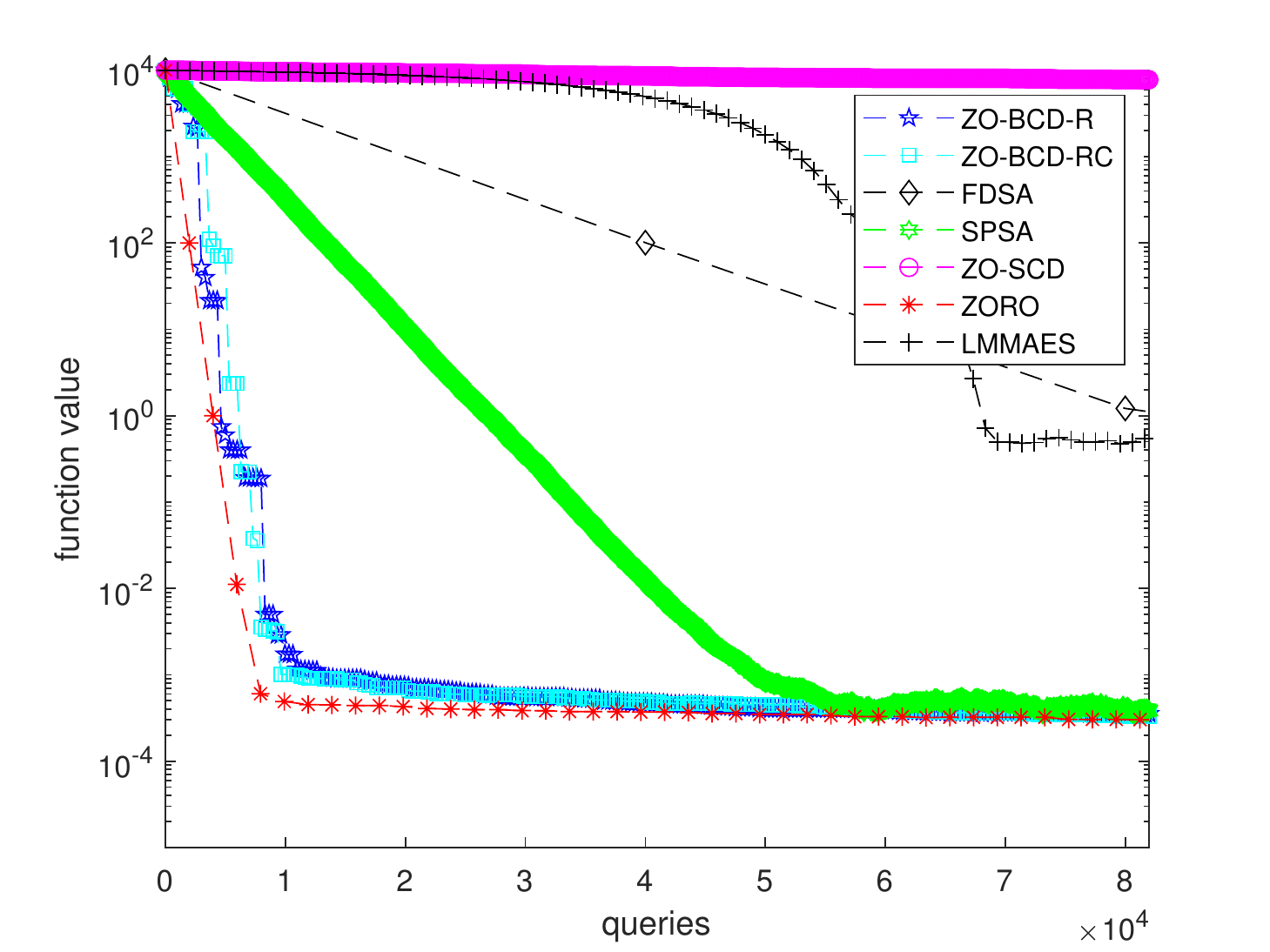} \label{fig:quad}}
\\
\vspace{-0.05in}
\subfloat[Max-$s$-squared-sum.]{
\includegraphics[width = 0.95\linewidth]{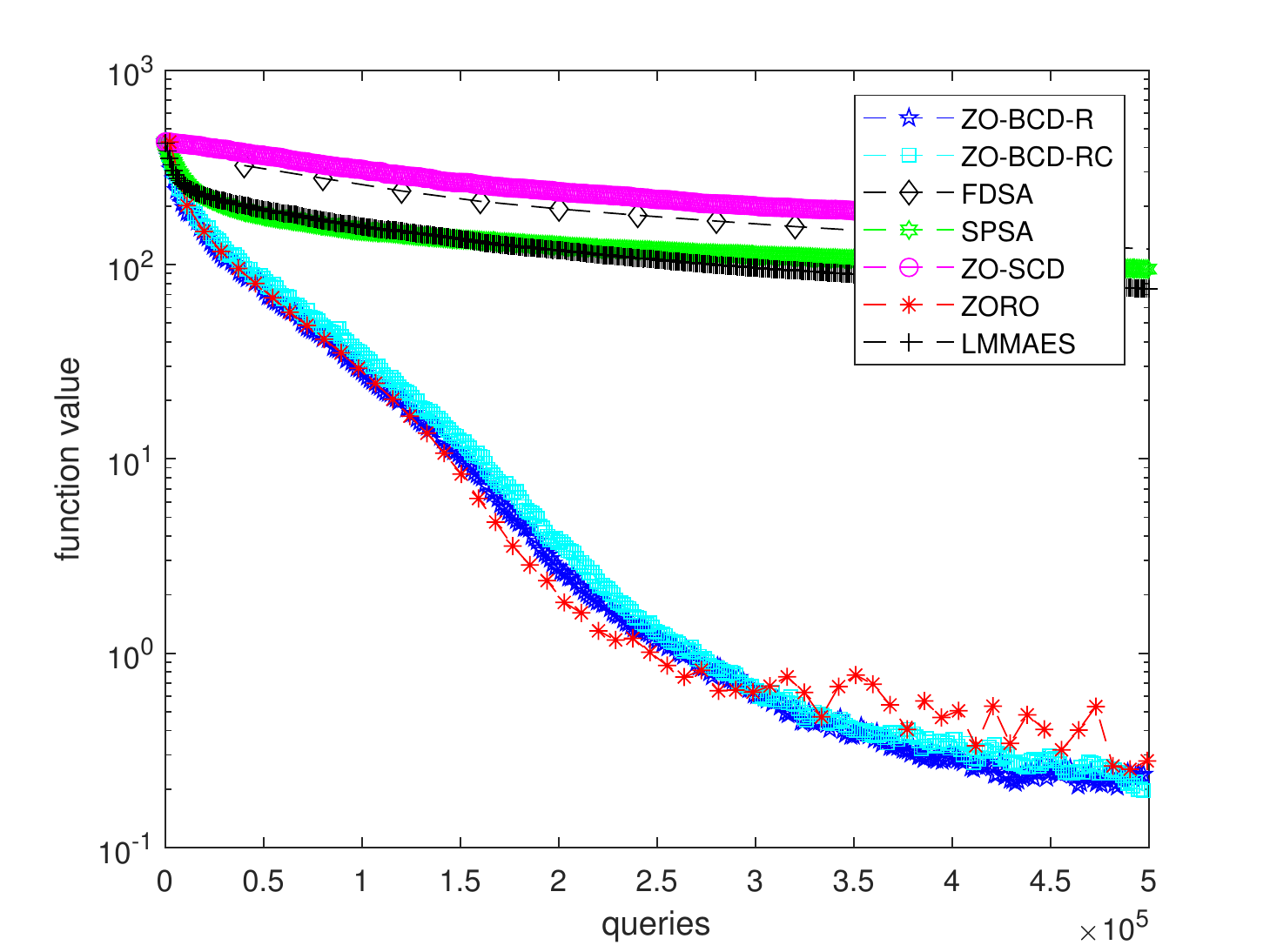} \label{fig:maxk}}
\caption{Function values  \textit{v.s.} queries for ZO-BCD (-R and -RC, with $5$ blocks) and four other representative zeroth-order methods. ZO-BCD is fast and stable (while running faster with less memory).} 
\label{fig:Algo_Comparisons}
\end{figure}

We study the performance of ZO-BCD with noisy oracles on the zeroth-order optimization problem $\minimize_{x\in\mathbb{R}^{d}} f(x)$ for two selected objective functions: 
\begin{itemize}
\vspace{-0.05in}
    \item[(a)] Sparse quadratic function: $f(x) =  \frac{1}{2}x^TAx$, where $A$ is a diagonal matrix with $s$ non-zero entries. 
    \item[(b)] Max-$s$-sum-squared function: $f(x)=\frac{1}{2}\sum_{m_i}^{s} x_{m_i}^2$, where $x_{m_i}$ is the $i$-th largest-in-magnitude entry of $x$. This problem is more complicated than (a) as $m_i$ changes with $x$.
\end{itemize}
We use $d=20,000$ and $s=200$ in both problems, so they have high ambient dimension with sparse gradients. 

As can be seen in Figure~\ref{fig:Algo_Comparisons}, both versions of ZO-BCD effectively exploit the gradient sparsity, and have very competitive performance in terms of queries. 
In particular, ZO-BCD converges more stably than the state-of-the-art ZORO in max-$s$-squared-sum problem while its computational and memory complexities are much lower. SPSA's query efficiency is roughly the same as that of ZO-BCD and ZORO when the gradient support does not change (see Figure~\ref{fig:quad}); however, it is {\em significantly worse} when the gradient support is allowed to change (see Figure~\ref{fig:maxk}).

\paragraph{Number of blocks.} We study the performance of ZO-BCD with different numbers of blocks. The numerical results are summarized in Figure~\ref{fig:varying number of blocks} and Table~\ref{tab:Run_time_varying_blk}. 
Note that the runtime of each query can vary a lot between problems, so we only count the empirical runtime excluding the time of making queries. 
As a rule of thumb, we find that using fewer blocks yields smoother convergence and a more accurate final solution (see Figure~\ref{fig:varying number of blocks}), at the cost of a higher runtime (see Table~\ref{tab:Run_time_varying_blk}).
This phenomenon matches our theoretical result in Theorem~\ref{thmNonReg}. Generally speaking, we recommend using a mild number of blocks to balance between ZO-BCD's speed and convergence performance.

\begin{figure}[t]
\centering
\subfloat[Sparse Quadric.]{
\includegraphics[width = 0.95\linewidth]{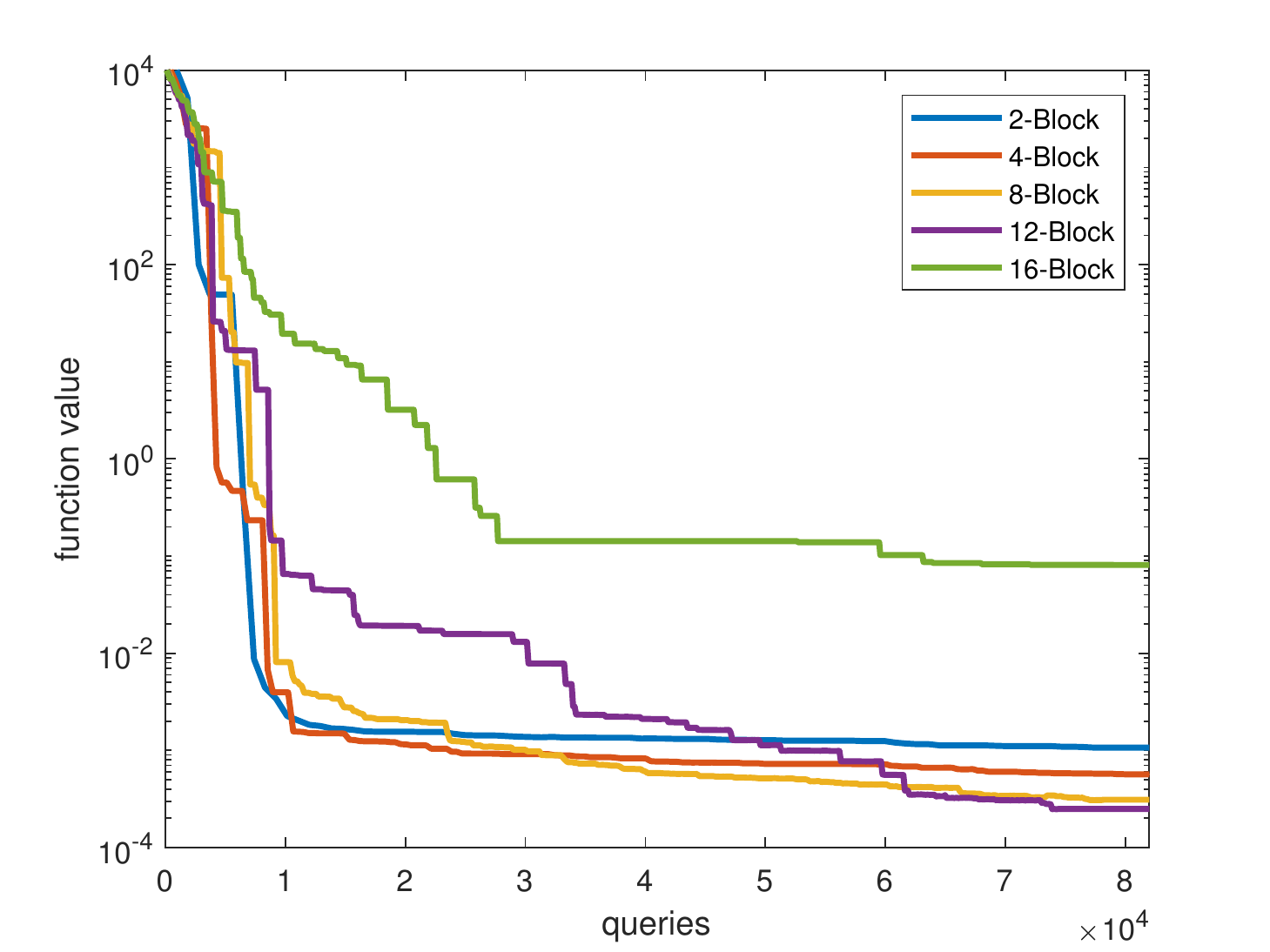} \label{fig:quad_varying_blks}}
\\
\vspace{-0.05in}
\subfloat[Max-$s$-squared-sum.]{
\includegraphics[width = 0.95\linewidth]{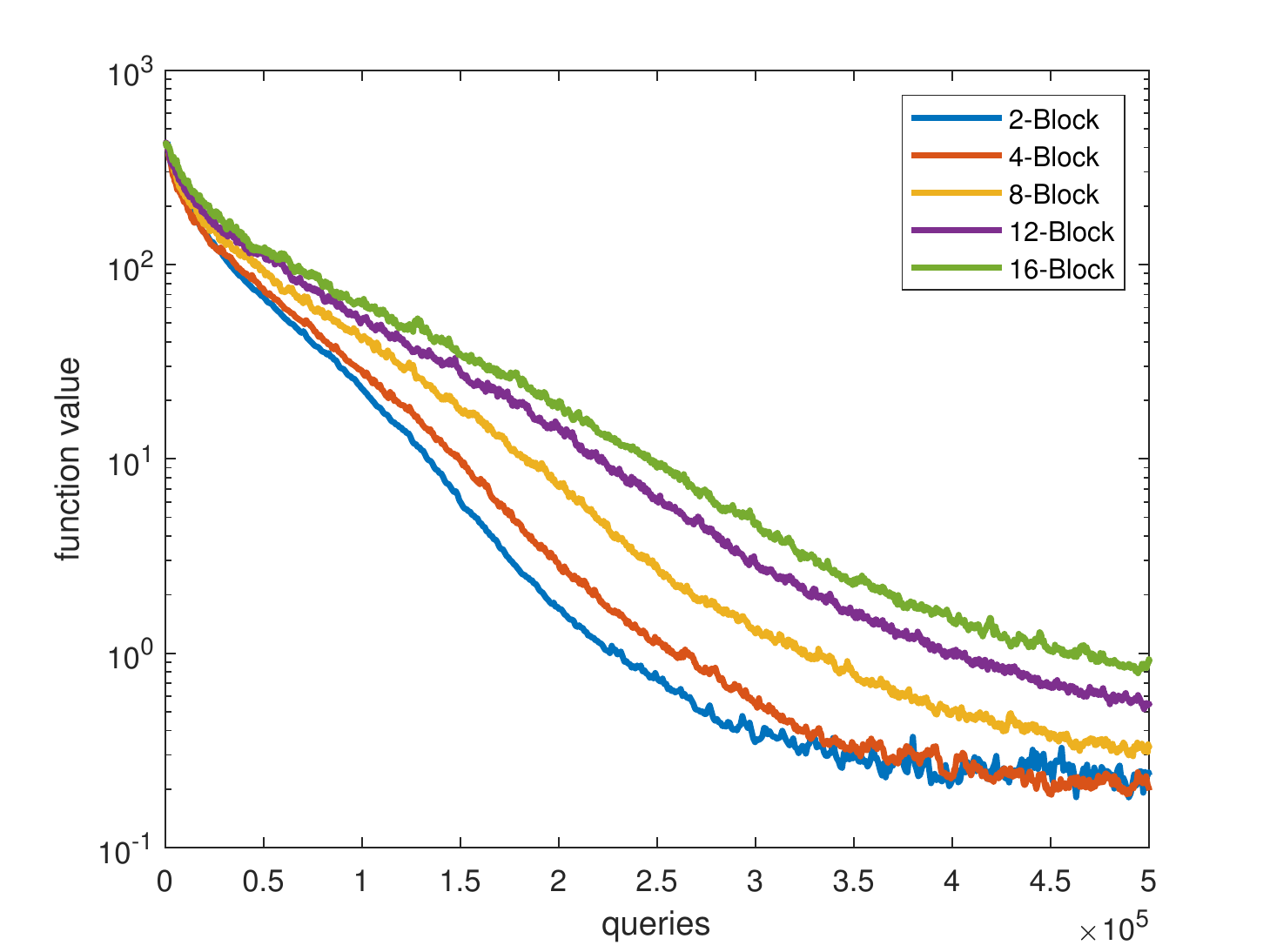} \label{fig:maxk_varying_blks}}
\caption{Function values \textit{v.s.} queries for ZO-BCD-R with different numbers of blocks.} 
\label{fig:varying number of blocks}
\end{figure}

\begin{table}[htb]
\caption{Runtime per iteration and the number of iterations to reach tolerance for ZO-BCD-R with varying number of blocks ($J$).} \label{tab:Run_time_varying_blk}
\vspace{0.15in}
\begin{center}
\begin{small}
\begin{tabular}{c|cc|cc}
\toprule
 & \multicolumn{2}{c|}{Sparse Quadric.}  & \multicolumn{2}{c}{Max-$s$-squared-sum}\\
$J$ & \textsc{Sec/Itr} & \textsc{Itr} to $10^{-2}$ & \textsc{Sec/Itr} & \textsc{Itr} to $10^{0}$ \\
\midrule
~~2  & .1093 sec  & 8 & .1362 sec  &  249\\
~~4  & .0244 sec  & 20 & .0358 sec  &  605\\
~~8  & .0054 sec  & 45 & .0116 sec  &  1651\\
12   & .0026 sec  & 224 & .0054 sec  &  3185\\
16   & .0019 sec  & N/A & .0042 sec  &  5090 \\ 
\bottomrule
\end{tabular}
\end{small}
\end{center}
\end{table}

\paragraph{Scalability.}
Since ZO-BCD and ZORO are the only methods that have competitive convergences in the query complexity experiments (see Figure~\ref{fig:Algo_Comparisons}), it is only meaningful to compare their computational complexities with respect to problem dimension $d$. We record the runtime of ZO-BCD and ZORO for solving the sparse quadratic function with varying problem dimensions, where the stopping condition is set to be $f(x_k)\leq 10^{-2}$ for all tests. Similar to the experiment of varying numbers of blocks, we only count the empirical runtime excluding the time of making queries in this experiment. 
The test results are presented in Figure~\ref{fig:time_vs_dim}, where one can see that ZO-BCD-R has significant speed advantage over ZORO  
and ZO-BCD-RC is even faster, especially when problem dimension is large. 

\begin{figure}[t]
    \centering
    \subfloat[Sparse Quadric.]{\includegraphics[width = 0.95\linewidth]{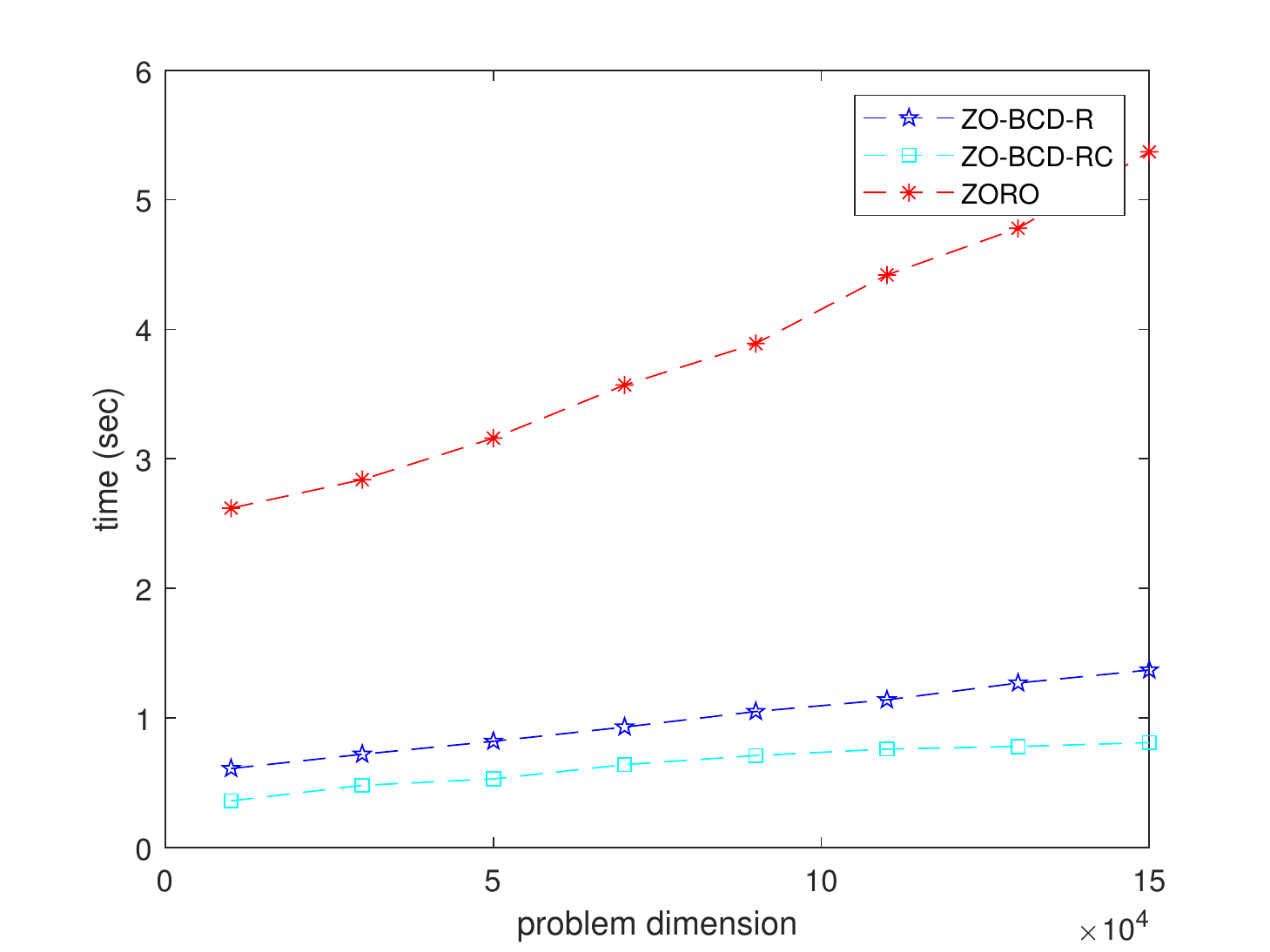}}
    \caption{Runtime \textit{v.s.} problem dimension for ZO-BCD (-R and -RC, with $5$ blocks) and ZORO.}
    \label{fig:time_vs_dim}
\end{figure}

\subsection{Sparse DWT attacks on images}
\label{sec:ImageAttack}
\begin{table}[t]
\caption{Results of untargeted adversarial attack on images using various zeroth-order algorithms. Attack success rate (ASR), average final $\ell_2$ distortion (on pixel domain), average iterations and number of queries till 1st successful attack. ZO-BCD-R(large coeff.) stands for applying ZO-BCD-R to attack only large wavelet coefficients ($\mathrm{abs}\geq0.05$).} \label{tab:attack_result}
\vspace{0.1in}
\begin{center}
\begin{small}
\begin{tabular}{lccc}
\toprule
\textsc{Methods} & \textsc{ASR} & \textsc{$\ell_2$ dist}  & \textsc{Queries}\\
\midrule
ZO-SCD    & 78$\%$  &  57.5 &  2400\\
ZO-SGD    & 78$\%$  &  37.9 &  $\textbf{1590}$\\
ZO-AdaMM  & 81$\%$  &  28.2 &  1720\\
ZORO      & 90$\%$  &  21.1 &  2950 \\
ZO-BCD-R  & 92$\%$  &  14.1 &  2131 \\
ZO-BCD-RC & 92$\%$  &  14.2 &  2090\\
ZO-BCD-R(large coeff.) & $\textbf{96}\%$  &  $\textbf{13.7}$ &   1662 \\
\bottomrule
\end{tabular}
\end{small}
\end{center}
\vspace{-0.05 in}
\end{table}

\begin{table}[h]
\begin{center}
\caption{Results of sparse DWT adversarial attack on images, using ZO-BCD-R and different values of $s$ (sparsity) and $J$ (number of blocks). ZO-BCD-R is robust to the particular choice of $s$ and $J$. Note $d = 676,353$ is the problem dimension.}\label{tab:attack_robustness}
\vspace{0.15in}
\begin{small}
\begin{tabular}{lccc}
\toprule
\textsc{ZO-BCD-R} & \textsc{ASR} & \textsc{$\ell_2$ dist} & \textsc{Queries} \\
\midrule
$J=2000$, $s=0.05d$ & \textbf{93}\%  &  15.3 & 2423\\
$J=4000$, $s=0.05d$ & 91\%  &  14.0 & 2109\\
$J=8000$, $s=0.05d$  & \textbf{93}\%  &  14.8 & 2145\\
$J=12000$, $s=0.05d$ & 90\%  &  13.8 & \textbf{1979}\\
$J=4000$, $s=0.01d$  & 86\% &   22.6 & 2440\\
$J=4000$, $s=0.025d$ & \textbf{93}\% &  16.5 & 2086\\ 
$J=4000$, $s=0.1d$   &  90\% &   \textbf{13.1} & 2364\\ 
\bottomrule
\end{tabular}
\end{small}
\end{center}
\end{table}

\begin{figure}[t]
\centering
\subfloat[ZO-BCD-R: ``barbershop'' $\rightarrow$ ``flagpole'']{\includegraphics[width=.47\linewidth]{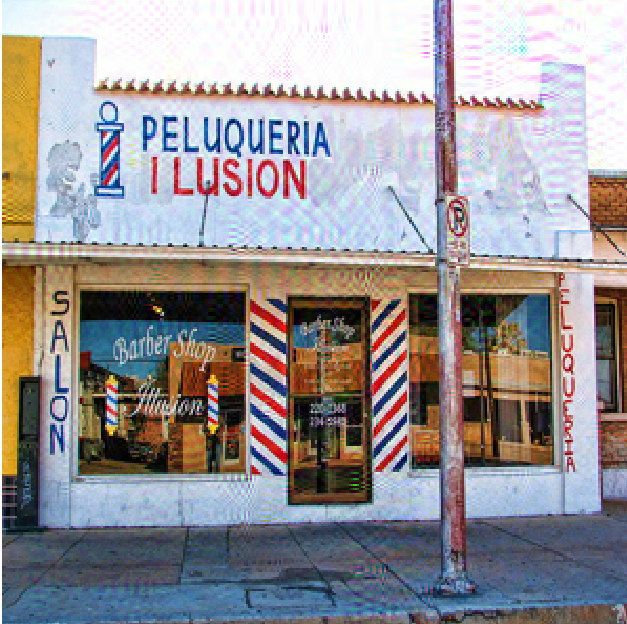}} 
\hfill
\subfloat[ZO-BCD-R (large coeff.): ``barbershop'' $\rightarrow$ ``flagpole'']{\includegraphics[width=.47\linewidth]{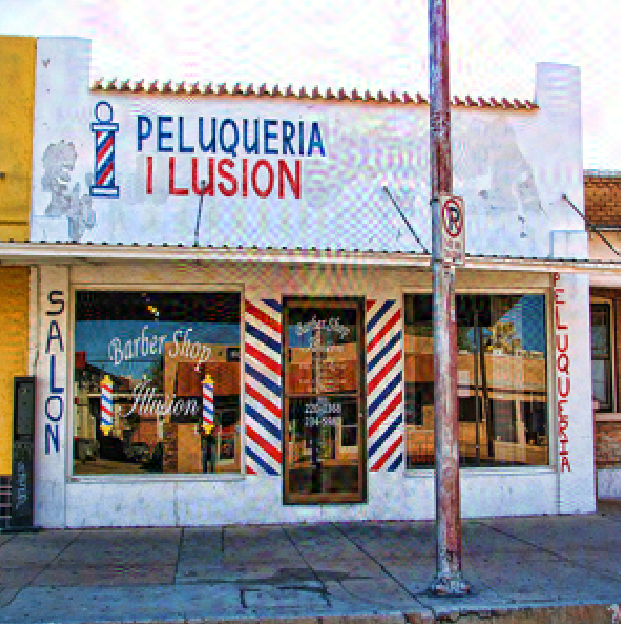}}
\\
\subfloat[ZO-BCD-R: ``dumbbell'' $\rightarrow$ ``computer keyboard'']{\includegraphics[width=.47\linewidth]{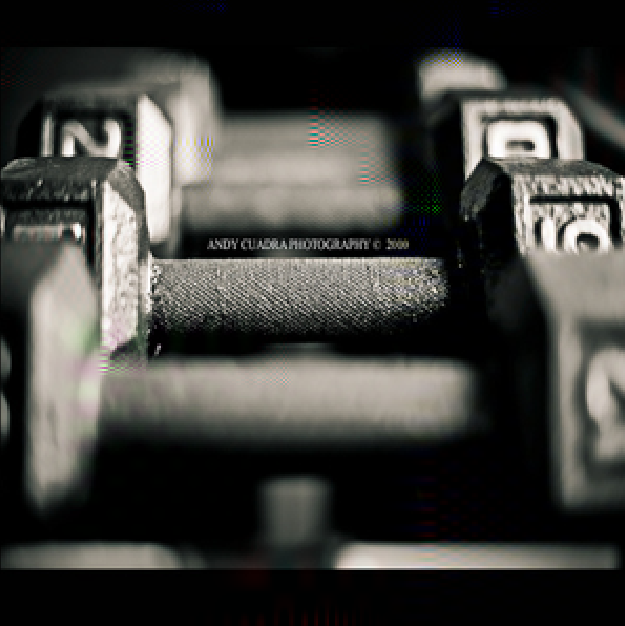}}
\hfill
\subfloat[ZO-BCD-R (large coeff.): ``dumbbell'' $\rightarrow$ ``computer keyboard'']{\includegraphics[width=.47\linewidth]{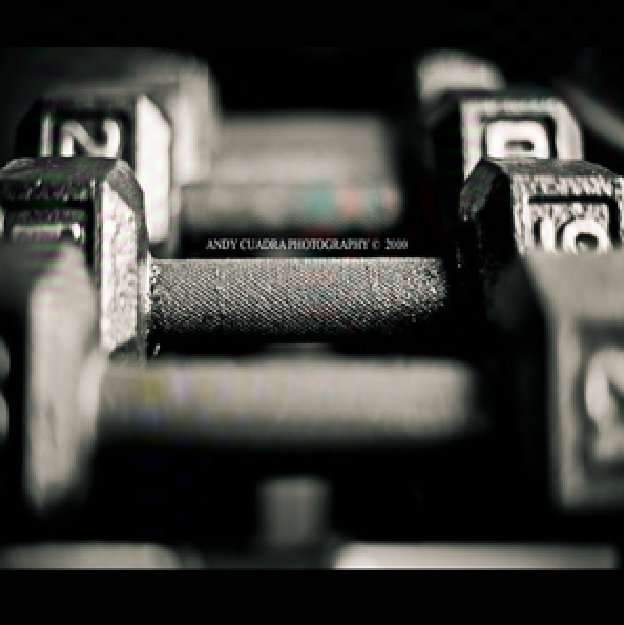}}
\vspace{-0.05in}
\caption{Examples of wavelet attacked images by ZO-BCD-R and ZO-BCD-R (attack restricted to only large coefficients, \textit{i.e.} $\textrm{abs}\geq 0.05$), true labels and mis-classified labels.}
\label{fig:AttackedImages}
\end{figure}

We consider a wavelet domain, untargeted, per-image attack on the \texttt{ImageNet} dataset \citep{deng2009imagenet} with the pre-trained \texttt{Inception-v3} model \citep{szegedy2016rethinking}, as discussed in Section~\ref{sec:sparse_attack}. We use the famous `db45' wavelet \citep{daubechies1992ten} with $3$-level DWT in these attacks.
Empirical performance was evaluated by attacking $1000$ randomly selected \texttt{ImageNet} pictures that were initially classified correctly. In addition to full wavelet domain attacks using both ZO-BCD-R and ZO-BCD-RC, we experiment with only attacking large wavelet coefficients, \textit{i.e.} the important components of the images in terms of the wavelet basis. If we only attack wavelet coefficients greater than $0.05$ in magnitude, the problem dimension is reduced by an average of $67.3\%$ for the tested images; nevertheless, the attack problem dimension is still as large as $\sim90,000$, so ZO-BCD is still suitable for this attack problem.

 The test results are summarized in Table~\ref{tab:attack_result}. All three versions of ZO-BCD wavelet attack beat the other state-of-the-art methods in both attack success rate and $\ell_2$ distortion, and the large-coefficients-only (\textit{i.e.} we only attack wavelet coefficients with $\mathrm{abs}\geq0.05$) wavelet attack by ZO-BCD-R achieves the best results. Furthermore, ZO-BCD is robust to the choice of the number of blocks and sparsity, as summarized in Table~\ref{tab:attack_robustness}. We present a few visual examples of these adversarial attacks in Figure~\ref{fig:AttackedImages}. More examples, and detailed experimental settings, can be found in Appendix~\ref{sec:exp setup detail}.

\subsection{Sparse CWT attacks on audio signals}
\label{sec:AudioAttack}



\begin{figure}[h]
\vspace{0.15in}
\centering
\subfloat[Attack success rate.]{
\includegraphics[width = 0.98\linewidth]{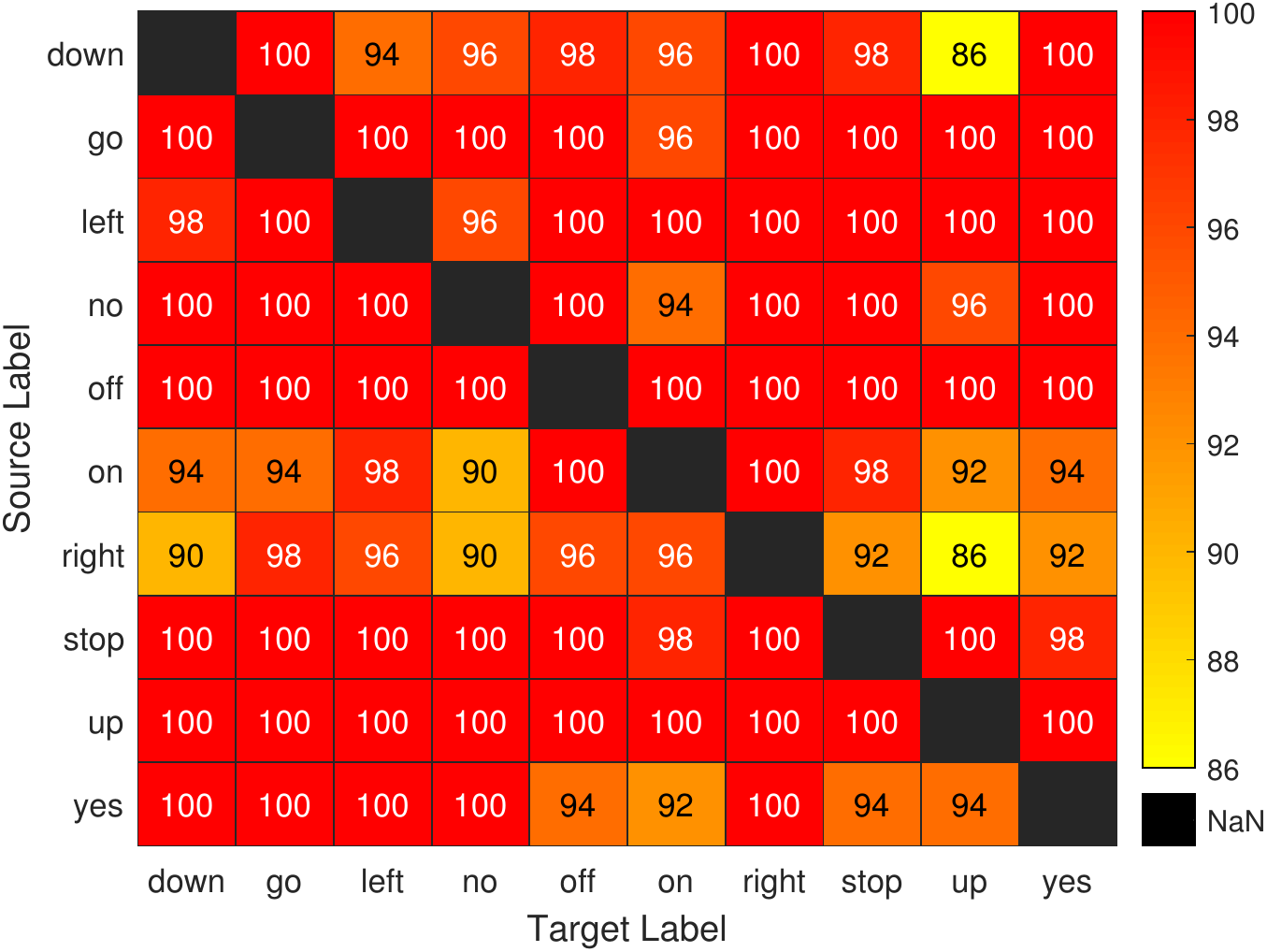} \label{fig:AudioASR}}
\\
\vspace{0.05in}
\subfloat[Relative loudness, in decibels.]{
\includegraphics[width = 0.98\linewidth]{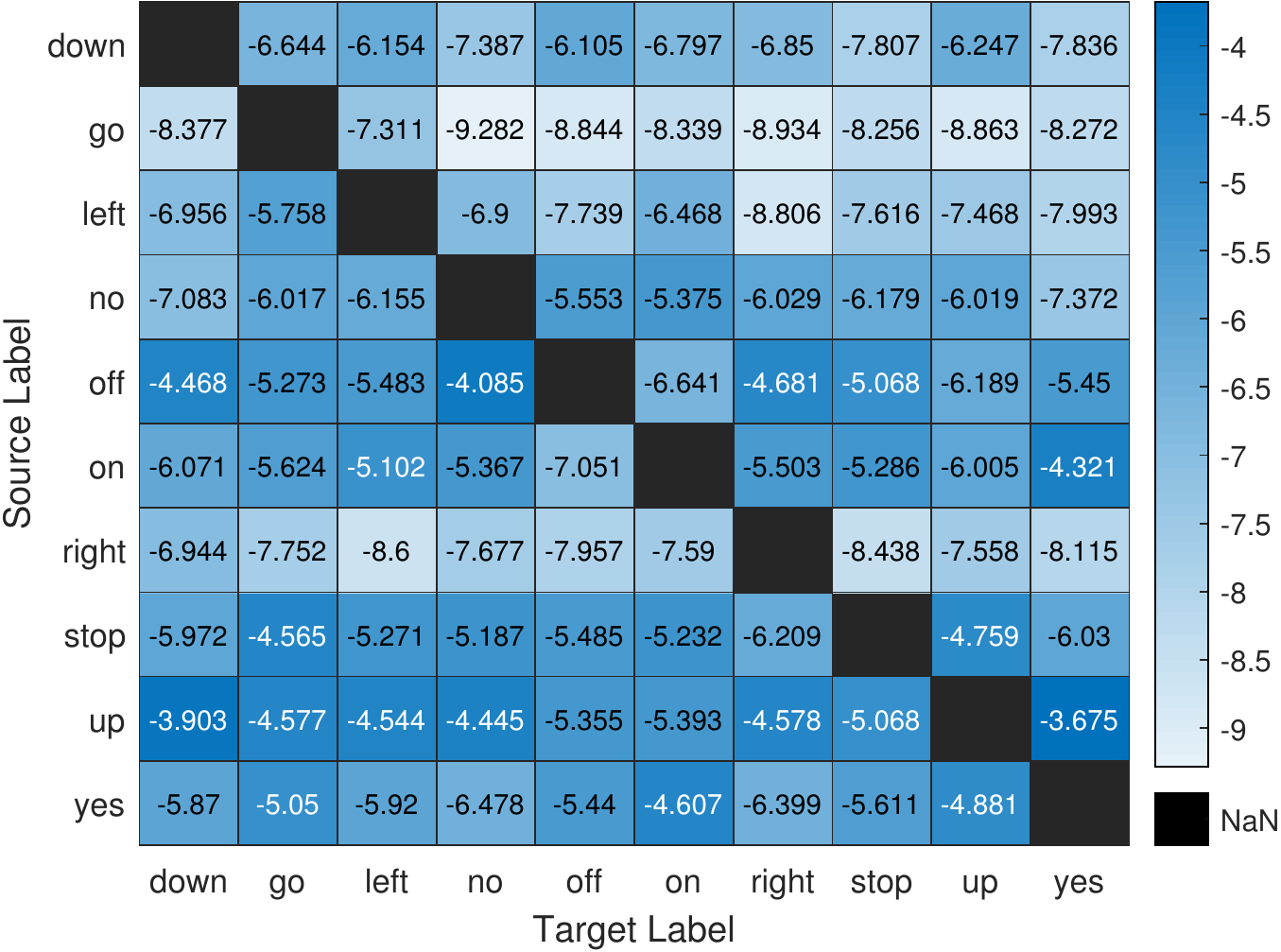} \label{fig:AudiodB}}
\caption{Detailed results for \textit{targeted} sparse wavelet attacks on audio signals.} \label{fig:AudioASR_and_dB}
\end{figure}

\begin{table}[t]
\caption{Results for \textit{untargeted} attacks on audio signals using ZO-BCD-R in the time domain, in the wavelet domain using a step-size of $0.02$ and in the wavelet domain using a step-size of $0.05$. Attack success rate (ASR), average final Decibel distortion and average number of queries to 1st successful attack.}
\label{tab:untargeted_audio}
\vspace{0.1in}
\begin{center}
\begin{small}
\begin{tabular}{lrrc}
\toprule
\textsc{Domains} & \textsc{ASR}  & d\textsc{B dist} & \textsc{Queries}\\
\midrule
Time    & $\textbf{100}\%$  & +1.5597 & $\textbf{894}$ \\
Wavelet (0.02)  & 99.9$\%$ & -$\textbf{13.8939}$ & 3452 \\
Wavelet (0.05)  & $\textbf{100}\%$ & -7.1192 & 2502 \\
\bottomrule
\end{tabular}
\end{small}
\end{center}
\end{table}

\begin{table}[t]
\caption{Results of attacks on {\tt SpeechCommands} dataset. A = \citep{alzantot2018did}, V\&S = \citep{vadillo2019universal}, Li = \citep{li2020advpulse}, Xie = \citep{xie2020enabling}. Univ. = Universal.} \label{tab:attack_result_audio}
\vspace{0.1in}
\begin{center}
\begin{small}
\begin{tabular}{lcccc}
\toprule
\textsc{Method} & \textsc{ASR} & \textsc{Univ.} & \textsc{Black-box} & \textsc{Targeted} \\
\midrule
A. & $89.0\%$ & \textsc{No} & \textsc{Yes} & \textsc{Yes} \\
V\&S & $70.4\%$ & \textsc{Yes} & \textsc{No} & \textsc{No} \\
Li & $96.8\%$ & \textsc{Yes} & \textsc{No}  & \textsc{Yes} \\
Xie & $97.8\%$ & \textsc{No} & \textsc{No}  & \textsc{Yes} \\ 
ZO-BCD  & $\textbf{97.9}\%$ & \textsc{No} & \textsc{Yes} & \textsc{Yes} \\
\bottomrule
\end{tabular}
\end{small}
\end{center}
\end{table}


We consider targeted per-clip audio adversarial attacks on the \texttt{SpeechCommands} dataset \citep{warden2018speech}, which consists of $1$-second audio clips, each containing a one-word voice command, \textit{e.g.} ``yes'' or ``left''. The audio sampling rate is $16$kHz thus each clip is a $16,000$ real valued vector. Adversarial attacks against this data set have been considered in \citep{alzantot2018did,vadillo2019universal,li2020advpulse} and \citep{xie2020enabling}, although with the exception of \citep{alzantot2018did} all these works consider a white-box setting\footnote{There are other subtle differences in the threat models considered in these works, as compared to ours (see Appendix~\ref{appendix:Audio}).}.
The victim model is a pre-trained, 5 layer, convolutional network called \texttt{commandNet} \citep{commandNet}. The architecture is essentially as proposed in \citep{sainath2015convolutional}. It takes as input the bark spectrum coefficients of a given audio clip, a transform closely related to the Mel Frequency transform. The test classification accuracy of this model (on un-attacked audio clips) is $94.46\%$. We use the Morse \citep{olhede2002generalized} continuous wavelet transform with $111$ frequencies, resulting in a problem dimension of $111\times 16,000 = 1,776,000$. As discussed in \citep{carlini2018audio}, the appropriate measure of size for the attacking distortion $\delta$ is relative loudness:
\begin{equation*}
    \mathrm{dB}_{x}(\delta) := 20\left(\max_{i}\log_{10}(|x_i|) - \max_{i}\log_{10}(|\delta_i|) \right).
\end{equation*}
The results are detailed in Table~\ref{tab:attack_result_audio} and Figure~\ref{fig:AudioASR_and_dB}. Overall, we achieve a $97.93\%$ ASR using a mean of $7073$ queries. Our attacking distortions have a mean volume of $-6.32 \mathrm{dB}$. As can be seen, our proposed attack exceeds the state of the art in attack success rate (ASR), surpassing even white-box attacks! This is not to claim that our proposed method is strictly better than others, as there are multiple factors to consider when judging the ``goodness'' of an attack (ASR, attack distortion, universality etc.), see Appendix~\ref{appendix:Audio}. The attacking noise can be heard as a slight ``hiss'', or white noise in the attacked audio clips. The original keyword however is easy for a human listener to make out. We encourage the reader to listen to a few examples, available at \url{https://github.com/YuchenLou/ZO-BCD}. 

Out of curiosity, we also tested using ZO-BCD to craft untargeted adversarial attacks in the time domain ({\em i.e.} without using a wavelet transform) for 1000 randomly selected audio clips. The results are underwhelming; indeed the attacking perturbation is on average significantly {\em louder} than the victim audio clip (see Table~\ref{tab:untargeted_audio})! This suggests attacking in a wavelet domain is much more effective than attacking in the original signal domain.

\section{Conclusion}
We have introduced ZO-BCD, a novel zeroth-order optimization algorithm. ZO-BCD enjoys strong, albeit probabilistic, convergence guarantees. We have also introduced a new paradigm in adversarial attacks on classifiers: the sparse wavelet domain attack. On medium-scale test problems the performance of ZO-BCD matches or exceeds that of state-of-the-art zeroth order optimization algorithms, as predicted by theory. However, the low per-iteration computational and memory requirements of ZO-BCD means that it can tackle huge-scale problems that for other zeroth-order algorithms are intractable. We demonstrate this by successfully using ZO-BCD to craft adversarial examples, to both image and audio classifiers, in wavelet domains where the problem size can exceed 1.7 million.

\section*{Acknowledgements}
We thank Zhenliang Zhang for valuable discussions on the related work. We also thank the anonymous reviewers for their helpful feedback.

\bibliographystyle{icml2021}
\bibliography{BlockZOROBib}

\newpage

\appendix

\section{Proofs for Section~\ref{sec:grad estimator}}

\begin{proof}[Proof of Theorem~\ref{thm:Grad_Estimate_Error}]
Consider the function $f^{(j)}(t) = f(x+ U^{(j)}t)$. This function has gradient $g^{(j)}$ and Hessian $\nabla^{2}_{jj}f$. By Assumption~\ref{assumption:Sparsity}  $\|g^{(j)}\|_{0} \leq \|g\|_{0} \leq s$ while $\|\nabla^{2}_{jj}f\|_{1} \leq H$  from Assumption~\ref{assumption:WeakSparsity}. Thus, $f^{(j)}(t)$ satisfies the assumptions of Corollary~2.7 of \citep{cai2020zeroth} so Theorem~\ref{thm:Grad_Estimate_Error} follows from this result.
\end{proof}

\begin{proof}[Proof of Theorem~\ref{thm:equalsparsity}]
For notational convenience, for this proof only we let $s:= s_{\mathrm{exact}}$. Let $g_{i_1},\ldots, g_{i_k},\ldots, g_{i_s}$ denote the non-zero entries of $g := g(x)$, after the permutation has been applied. Let $Y_{j}$ be the random variable counting the number of $g_{i_k}$ within block $j$:
\begin{equation*}
Y_{j} = \#\{i_k: \ i_k \textnormal{ in block } j\}.
\end{equation*}
Thus, the random vector $\mathbf{Y} = (Y_1,\ldots, Y_{J}) \in\mathbb{R}^{J}$ obeys the multinomial distribution. Observe $\sum_{j=1}^{J}Y_{j} = s$ and, because the blocks are equally sized, $\mathbb{E}(Y_j) = \frac{s}{J}$.

By Chernoff's bound, for $\Delta>0$ and each $Y_{j}$ individually:
\begin{equation*}
\mathbb{P}\left[\left| Y_{j} - \frac{s}{J}\right| \geq \Delta\frac{s}{J}\right] \leq 2e^{-\frac{\Delta^2\mathbb{E}(Y_j)}{3}}=2e^{-\frac{\Delta^2s}{3J}}.
\end{equation*}
Applying the union bound:
\begin{align*}
& \quad~\mathbb{P}\left[ \exists \ j \textnormal{ s.t. } \left| Y_{j} - \frac{s}{J}\right| \geq \Delta\frac{s}{J}\right] \\
& \leq \sum_{j=1}^{J}\mathbb{P}\left[ \left| Y_{j} - \frac{s}{J}\right| \geq \Delta\frac{s}{J} \right] \leq 2Je^{-\frac{\Delta^2s}{3J}},
\label{eq:Exists_forAll}
\end{align*}
and so:
\begin{align*}
&\quad~\mathbb{P}\left[ |Y_{j} - \frac{s}{J}| \leq \Delta\frac{s}{J},\,\, \forall j\in [1,\cdots,J]\right]\\
&= 1 - \mathbb{P}\left[ \exists \ j \textnormal{ s.t. } \left| Y_{j} - \frac{s}{J}\right| \geq \Delta\frac{s}{J}\right] \\
& \geq 1 - 2Je^{-\frac{\Delta^2s}{3J}}.
\end{align*}
This finishes the proof.
\end{proof}

\begin{proof}[Proof of Corollary~\ref{cor:block error bound}]
From Theorem~\ref{thm:equalsparsity} we have that, with probability at least $1- 2J\exp(\frac{-0.01s_{\mathrm{exact}}}{3J})$, $\|g^{(j)}\|_{0} \leq 1.1s_{\mathrm{exact}}/J \leq s$. Assuming this is true, \eqref{eq:ErrorBound2} holds with probability $1 - (s/d)^{b_2 s}$. These events are independent, thus the probability that they both occur is:
\begin{align*}
 & \left(1- 2J\exp(\frac{-0.01s_{\mathrm{exact}}}{3J})\right)\left(1 - (s/d)^{b_2 s}\right)
\end{align*}
Because $s \ll d$ the term exponential in $s_{\mathrm{exact}}/J$ is significantly larger than $(s/d)^{b_2s}$. Expanding, and keeping only dominant terms we see that this probability is equal to $1 - \mathcal{O}\left(J\exp(\frac{-0.01s_{\mathrm{exact}}}{3J})\right)$.
\end{proof}

We emphasize that Corollary~\ref{cor:block error bound} holds for all $j$. Before proceeding, we remind the reader that for fixed $s$ the {\em restricted isometry constant} of $Z \in \mathbb{R}^{m\times n}$ is defined as the smallest $\delta > 0$ such that:
\begin{equation*}
    (1-\delta)\|v\|_{2}^{2} \leq \|Zv\|_{2}^{2} \leq (1+\delta)\|v\|_{2}^{2}
\end{equation*}
for all $v\in\mathbb{R}^{n}$ satisfying  $\|v\|_{0} \leq s$ .The key ingredient to the proof of Theorem~\ref{cor:Circulant} is the following result:

\begin{theorem}[{\citep[Theorem 1.1]{krahmer2014suprema}}]
\label{thm:kramer}
Let $z\in\mathbb{R}^{n}$ be a Rademacher random vector and choose a random subset $\Omega = \{j_1,\ldots, j_m\} \subset \{1,\ldots, n\}$ of cardinality $m = c\delta^{-2}s\log^2(s)\log^{2}(n)$. Let $Z \in \mathbb{R}^{m\times n}$ denote the matrix with rows $\frac{1}{\sqrt{m}}\mathcal{C}_{j_i}(z)$. Then $\delta_{s}(Z) < \delta$ with probability $1 - n^{-\log(n)\log^{2}(s)}$. 
\end{theorem}

$c$ is a universal constant, independent of $s,n$ and $\delta$. Similar results may be found in \citep{mendelson2018improved} and \citep{huang2018improved}. In \citep{krahmer2014suprema} a more general version of this theorem is provided, which allows the entries of $z$ to be drawn from any sub-Gaussian distribution.  
\begin{proof}[Proof of Theorem~\ref{cor:Circulant}]
Let $Z$ be the sensing matrix with rows $\frac{1}{\sqrt{m}}\mathcal{C}_{j_i}(z)$.
By appealing to Theorem~\ref{thm:kramer} with $s = 4s$, $n = d/J$ and $\delta = 0.3843$ we get $\delta_{4s}(Z) \leq 0.3843$, with probability $1 - \left(d/J\right)^{\log(d/J)\log^2(4s)}$
From Theorem~\ref{thm:equalsparsity} we have that, with probability at least $1- 2J\exp(\frac{-0.01s_{\mathrm{exact}}}{3J})$, $\|g^{(j)}\|_{0} \leq 1.1s_{\mathrm{exact}}/J \leq s$. Assuming $\delta_{4s}(Z) \leq 0.3843$, Theorem~\ref{cor:Circulant} follows by the same proof as in \citep[Corollary~2.7]{cai2020zeroth}. The events $\|g^{(j)}\|_{0} \leq s/J$ and $\delta_{4s}(Z) \leq 0.3843$ are independent, hence they both occur with probability:
\begin{align} \label{eq:Universal_Prob_Bound}
 & \Big(1- 2J\exp\left(\frac{-0.01s_{\mathrm{exact}}}{3J}\right)\Big)\left(1 -\left(d/J\right)^{\log(d/J)\log^2(4s)}\right) \cr
 & \geq 1 - 2J\exp\left(\frac{-0.01s_{\mathrm{exact}}}{3J}\right) - \left(d/J\right)^{\log(d/J)\log^2(4s)}. 
\end{align}
Note that the third term in the right side of \eqref{eq:Universal_Prob_Bound} is {\em universal}, {\em i.e.} it holds for all $x \in \mathbb{R}^{d}$, as it depends only on the choice of $Z$, not the choice of $x$. 
\end{proof}




\section{ZO-BCD for unequally-sized blocks}
\label{sec:UnequalBlocks}

Using randomly assigned, equally-sized blocks is an important part of the ZO-BCD framework as it allows one to consider a block sparsity $\approx s/J$, instead of the worst case sparsity of $s$. Nevertheless, there may be situations where it is preferable to use user-defined, unequally-sized blocks. For such cases we recommend the following (we discuss the modifications here for ZO-BCD-R, but with obvious changes it also applies to ZO-BCD-RC). Let $s^{(j)} \leq s$ be an upper estimate of the sparsity of the $j$-th block gradient: $\|g^{(j)}(x)\|_{0} \leq s^{(j)}$. Let $m^{(j)} = b_1s^{(j)}\log(D/J)$ (and use the analogous formula for ZO-BCD-RC) and define $m^{\max} = \max_{j} m^{(j)}$. When initializing ZO-BCD-R, generate $m^{\max}$ Rademacher random variables: $z_{1},\ldots, z_{m^{\max}} \in \mathbb{R}^{d^{\max}}$. At each iteration, if block $j$ is selected, randomly select $i_{1},\ldots, i_{m^{(j)}}$ from $1,\ldots, m^{\max}$ and for $k=1,\ldots, m^{(j)}$ let $\tilde{z}_{i_{k}} \in \mathbb{R}^{d^{(j)}}$ denote the vector formed by taking the first $d^{(j)}$ components of $z_{i_k}$. Use $\{\tilde{z}_{i_{k}}\}_{k=1}^{m^{(j)}}$ as the input to Algorithm~\ref{algo:grad estimate}. Note that the $\{\tilde{z}_{i_{k}}\}_{k=1}^{m^{(j)}}$ possess the same statistical properties as the $\{z_{i_k}\}_{k=1}^{m^{(j)}}$ ({\em i.e.} they are i.i.d. Rademacher vectors) so using them as sampling directions will result in the same bound on $\|g^{(j)}_k - \hat{g}^{(j)}_k\|_2$ as Corollary~\ref{cor:block error bound}.

\section{Proofs for Section~\ref{sec:zo-bcd}}

Our proof utilizes the main result of \citep{tappenden2016inexact}. This paper requires $\alpha_{k} = \frac{1}{L_{j_k}}$, {\em i.e.} the step size at the $k$-th iteration is inversely proportional to the block Lipschitz constant. This is certainly ideal, but impractical. In particular, if the blocks are randomly selected it seems implausible that one would have good estimates of the $L_{j}$. Of course, since $L_{j} \leq L_{\max}$ we observe that $L_{\max}$ is {\em a} Lipschitz constant for every block, and thus we may indeed take $\alpha_{k} = \alpha = \frac{1}{L_{\max}}$. This results in a slightly slower convergence resulted, reflected in a factor of $L_{\max}$ in Theorem~\ref{thmNonReg}. Throughout this section we shall assume $f$ satisfies Assumptions~\ref{assumption:Lipschitz_Diff}--\ref{assumption:WeakSparsity}. For all $j=1,\ldots,J$ define:
\begin{align*}
    V_j(x,t) = \langle g^{(j)}(x),t \rangle+\frac{L_{\max}}{2}\|t\|^2_{2},
\end{align*}
so that, by Lipschitz differentiablity:
\begin{equation*}
    f(x_{k}+U^{(j)}t) \leq f(x_{k}) + V_j(x_{k},t).
\end{equation*}
Define $t^{*}_{k,j} := \argmin V_j(x_{k},t) := -\frac{1}{L_{\max}}g^{(j)}_k$ while let $t^{'}_{k,j}$ be the update step taken by ZO-BCD, {\em i.e.} $t^{'}_{k,j}= \frac{1}{L_{\max}}\hat{g}^{(j)}_k$. 

\begin{lemma}
\label{lemma:block lipschitz}
Suppose $f$ satisfies Assumptions~\ref{assumption:Lipschitz_Diff}--\ref{assumption:BoundedLevelSets}. Then $f(x)-f^*\geq \frac{1}{2L_{\max}}\| g^{(j)}(x) \|_2^2$ for any $x\in\mathbb{R}^{d}$ and any $j = 1,\ldots, J$.
\end{lemma}

\begin{proof}
Define $h_{x}:\mathbb{R}^{d^{(j)}}\to\mathbb{R}$ as $h_{x}(t) := f(x+U^{(j)}t)$ where $U^{(j)}$ is as described in Section~\ref{sec:Intro}. Since $U^{(j)}$ is a linear transformation and $f$ is convex, $h_{x}$ is also convex. By Assumption~\ref{assumption:Lipschitz_Diff} and $L_{j}\leq L_{\max}$, $h_{x}$ is $L_{\max}$-Lipschitz differentiable. From Assumption~\ref{assumption:BoundedLevelSets} it follows $\mathcal{Y}^{*} = \argmin_{t} h_{x}(t)$ is non-empty, and $h_{x}^{*} := \min_{t} h_{x}(t) \geq f^{*}$. Thus, from \citep[Proposition~B.3, part (c.ii)]{bertsekas1997nonlinear} , we have:
\begin{equation*}
h_{x}(t)-h_{x}^{*}\geq \frac{1}{2L_{\max}}\|\nabla h_{x}(t)\|_{2}^{2} = \frac{1}{2L_{\max}}\|g^{(j)}(x+U^{(j)}t)\|_{2}^{2}
\end{equation*}
for all $t$. Choose $t = 0$, and use $h_{x}(0) = f(x)$ and $f^{*} \leq h_{x}^{*}$ to obtain:
\begin{equation*}
f(x) - f^{*} \geq h_{x}(0) - h_{x}^{*} \geq  \frac{1}{2L_{\max}}\|g^{(j)}(x)\|_{2}^{2}.
\end{equation*}
This finishes the proof.
\end{proof}

\begin{lemma}
\label{lemma:EtaTheta}
Let $\eta = 2\rho^{2n}$ and $\theta = \frac{4\tau^{2}\sigma H}{L_{\max}}$. For each iteration of ZO-BCD
\begin{equation}
\label{eq:Inexactness}
 V_j(x_k,t')-V_j(x_k,t^*) \leq \eta(f(x_k)-f^*)+\theta.
\end{equation}
with probability $1 - \mathcal{O}\left(J\exp(-\frac{0.01s_{\mathrm{exact}}}{3J}\right)$ for ZO-BCD-R and with probability greater than
\begin{equation*}
    1 - 2J\exp\left(\frac{-0.01s_{\mathrm{exact}}}{3J}\right) - \left(d/J\right)^{\log(d/J)\log^2(4.4s/J)}
\end{equation*}
for ZO-BCD-RC.
\end{lemma}

\begin{proof}
For notational convenience, we define $t^{*} := t^{*}_{k,j}$, $t^{'} := t^{'}_{k,j}$ and $e_k^{(j)} := \hat{g}_k^{(j)}-g_k^{(j)}$. By definition:
\begin{equation*}
    t'_{k,j} = -\frac{1}{L_{\max}} \hat{g}_k^{(j)} = -\frac{1}{L_{\max}}(g_k^{(j)}+e_k^{(j)}).
\end{equation*}
Moreover:
\begin{align*}
    V_j(x_k,t^*) & = -\frac{1}{2L_{\max}}\| g_k^{(j)} \|_2^2  \\
    V_j(x_k,t') & = -\frac{1}{2L_{\max}}\| g_k^{(j)} \|_2^2+\frac{1}{2L_{\max}}\|e_k^{(j)} \|_2^2 ,
\end{align*}
and hence:
\begin{equation}
    V_j(x_k,t') - V_j(x_k,t^*) \leq \frac{1}{2L_{\max}}\|e_k^{(j)} \|_2^2 . \label{eq:ErrorAnalysis1}
\end{equation}
Recall $\hat{g}^{(j)}_k$ is the output of GradientEstimation with gradient sparsity level $s_{\mathrm{block}} = 1.1s/J$. Using the error bound \eqref{eq:ErrorBound2} and replacing $s$ with $1.1s/J$:
\begin{align} \label{eq:ErrorAnalysis2}
  \|e_k^{(j)} \|_2^2 & \leq \left( \rho^n\| g_k^{(j)} \|_{2} + 2\tau\sqrt{\sigma H} \right)^2 \nonumber \\
    & \leq 2\rho^{2n}\| g_k^{(j)} \|_{2}^{2} + 4\tau^{2}\sigma H. 
\end{align}
with the stated probabilities, as justified by Corollary~\ref{cor:block error bound} (for ZO-BCD-R) and Theorem~\ref{cor:Circulant} (for ZO-BCD-RC) Finally, from Lemma \ref{lemma:block lipschitz} $\| g_k^{(j)} \|_{2}^{2} \leq 2L_{\max}\left(f(x_k) - f^{*}\right)$ for any $j=1,\cdots,J$. Connecting this estimate with \eqref{eq:ErrorAnalysis1} and \eqref{eq:ErrorAnalysis2} we obtain:
\begin{equation*}
    V_j(x_k,t') - V_j(x_k,t^*) \leq \underbrace{2\rho^{2n}}_{=\eta} \left(f(x_k) - f^{*}\right) + \underbrace{\frac{4\tau^{2}\sigma H}{L_{\max}}}_{=\theta}.
\end{equation*}
This finishes the proof.
\end{proof}

\begin{proof}[Proof of Theorem~\ref{thmNonReg}]
Let $p_{j}$ denote the probability that the $j$-th block is chosen for updating at the $k$-th iteration. Because ZO-BCD chooses blocks uniformly at random, $p_{j} = 1/J$ for all $j$. If \eqref{eq:Inexactness} holds for all $k$ then by \citep[Theorem~6.1]{tappenden2016inexact} if:
\begin{align*}
    & \eta^2+\frac{4\theta}{c_1}<1 \textnormal{ where } c_1=2JL_{\max}\mathcal{R}^2(x_0), \\
    & \frac{c_1}{2}(\eta+\sqrt{\eta^2+\frac{4\theta}{c_1\zeta}})<\varepsilon < f(x_0) - f^{*}, \\
    & u := \frac{c_1}{2}\left(\eta+\sqrt{\eta^2+\frac{4\theta}{c_1}}\right), \\
    & K := \frac{c_1}{\varepsilon-u}+\frac{c_1}{\varepsilon-\eta c_1}\log\left(\frac{\varepsilon-\frac{\theta c_1}{\varepsilon-\eta c_1}}{\varepsilon\zeta-\frac{\theta c_1}{\varepsilon-\eta c_1}}\right)+2,
\end{align*}
then $\mathbb{P}(f(x_K)-f^*\leq\varepsilon)\geq 1-\zeta$. Note that:
\begin{itemize}
    \item Our $\eta$ and $\theta$ are $\alpha$ and $\beta$ in their notation.
    \item In \citep{tappenden2016inexact} $c_1 = 2\mathcal{R}^{2}_{\ell p^{-1}}(x_0)$ where $\mathcal{R}^{2}_{\ell p^{-1}}(x_0)$is defined as in \eqref{eq:level_set_radius} but using a norm $\|\cdot\|_{\ell p^{-1}}$ instead of $\|\cdot\|_{2}$. These norms are related as:
    \begin{align*}
        \|x\|_{\ell p^{-1}}^{2} &= \sum_{j=1}^{J} \frac{L_{j}}{p_j}\|x^{(j)}\|_{2}^{2} \stackrel{(a)}{=} \sum_{j=1}^{J} JL_{\max}\|x^{(j)}\|_{2}^{2} \\
        & = JL_{\max}\|x\|_{2}^{2}
    \end{align*}
    where (a) follows as $p_{j} = 1/J$ for all $j$ and we are taking $L_{j} = L_{\max}$. Hence, $c_1 = 2JL_{\max}\mathcal{R}^{2}(x_0)$ as stated. 
    \item $K = \tilde{\mathcal{O}}\left(J/\varepsilon\right)$.
\end{itemize}
Replace $\eta$ and $\theta$ with the expressions given by Lemma~\ref{lemma:EtaTheta} to obtain the expressions given in the statement of Theorem~\ref{thmNonReg}. Because \eqref{eq:Inexactness} holds with high probability at each iteration, by the union bound it holds for all $K$ iterations with probability greater than
\begin{align*}
    & 1 - K\mathcal{O}\left(J\exp\left(-\frac{0.01s_{\mathrm{exact}}}{3J}\right)\right) \\
    = & 1 - \tilde{\mathcal{O}}\left(\frac{J^2}{\varepsilon}\exp\left(-\frac{0.01s_{\mathrm{exact}}}{3J}\right)  \right)
\end{align*}
for ZO-BCD-R and with probability greater than
\begin{small}
\begin{align} \label{eq:Dont_Multiply}
& 1 - 2JK\exp\left(\frac{-0.01s_{\mathrm{exact}}}{3J}\right) - \left(d/J\right)^{\log(d/J)\log^2(4.4s/J)} \cr
& = 1 - \tilde{\mathcal{O}}\left(\frac{J^2}{\varepsilon}\exp\left(-\frac{0.01s_{\mathrm{exact}}}{3J}\right)\right) - \left(d/J\right)^{\log(d/J)\log^2(4.4s/J)} 
\end{align}
\end{small}
Note the third term in \eqref{eq:Dont_Multiply} is universal, {\em i.e.} holds for all iterations, and hence is not multiplied by $K$. In both cases we use $K =\tilde{\mathcal{O}}(J/\varepsilon)$. As stated, if \eqref{eq:Inexactness} holds for all $k$ then $\mathbb{P}(f(x_K)-f^*\leq\varepsilon)\geq 1-\zeta$. Apply the union bound again to obtain the probabilities given in the statement of Theorem~\ref{thmNonReg}.
ZO-BCD-R uses $m = 1.1b_1(s/J)\log(d/J)$ queries per iteration, all made by the GradientEstimation subroutine,  and hence makes:
\begin{align*}
    mK &= \left(1.1b_1(s/J)\log(d/J)\right)K = \tilde{\mathcal{O}}\left(s/\varepsilon\right)
\end{align*}
queries in total, using $K = \tilde{\mathcal{O}}\left(J/\varepsilon\right)$. On the other hand, ZO-BCD-RC makes $m = 1.1b_3(s/J)\log^2(1.1s/J)\log^{2}(d/J)$ queries per iteration. A similar calculation reveals that ZO-BCD-RC also makes $mK = \tilde{\mathcal{O}}\left(s/\varepsilon\right)$ queries in total.

The computational cost of each iteration is dominated by solving the sparse recovery problem using CoSaMP. It is known \citep{needell2009cosamp} that CoSaMP requires $\mathcal{O}(n\mathcal{T})$ FLOPS, where $\mathcal{T}$ is the cost of a matrix-vector multiply by $Z$. For ZO-BCD-R $Z\in\mathbb{R}^{m\times (d/J)}$ is dense and unstructured hence:
\begin{equation*}
\mathcal{T} = \mathcal{O}\left(m\frac{d}{J}\right) = \mathcal{O}\left(\frac{s}{J}\log(d/J)\frac{d}{J}\right) = \tilde{\mathcal{O}}\left(\frac{sd}{J^2}\right).
\end{equation*}
As noted in \citep{needell2009cosamp}, $n$ may be taken to be $\mathcal{O}(1)$ (In all our experiments we take $n \leq 10$). For ZO-BCD-RC, as $Z$ is a partial ciculant matrix, we may exploit a fast matrix-vector multiplication via fast Fourier transform to reduce the complexity to $\mathcal{O}(d/J\log(d/J))=\tilde{\mathcal{O}}(d/J)$. 

Finally, we note that for both variants the memory complexity of ZO-BCD is dominated by the cost of storing $Z$. Again, as $Z$ is dense and unstructured in ZO-BCD-R there are no tricks that one can exploit here, so the storage cost is $m(d/J) = \tilde{\mathcal{O}}(sd/J^2)$. For ZO-BCD-RC, instead of storing the entire partial circulant matrix $Z$, one just needs to store the generating vector $z\in\mathbb{R}^{d/J}$ and the index set $\Omega$. Assuming we are allocating $32$ bits per integer, this requires:
$$
\frac{d}{J} + 32\cdot b_3\frac{s}{J}\log^{2}\left(\frac{s}{J}\right)\log^{2}\left(\frac{d}{J}\right) = \mathcal{O}\left(\frac{d}{J}\right).
$$
This finishes the proof.
\end{proof}

\section{Experimental setup details}
\label{sec:exp setup detail}
In this section, we provide the detailed experimental settings for the numerical results provided in Section~\ref{sec:experiments}. 

\subsection{Settings for synthetic experiments}
For both synthetic examples, we use problem dimension $d=20,000$ and gradient sparsity $s=200$. Moreover, we use additive Gaussian noise with variance $\sigma=10^{-3}$ in the oracles. The sampling radius is chosen to be $10^{-2}$ for all tested algorithms. For ZO-BCD, we use $5$ blocks with uniform step size\footnote{We note that using a line search for each block would maximize the advantage of block coordinate descent algorithms such as ZO-BCD, but we did not do so here for fairness} $\alpha = 0.9$, and the per block sparsity is set to be $s_\textrm{block} = 1.05s/J=42$. Furthermore, the block gradient estimation step runs at most $n=10$ iterations of CoSaMP. For the other tested algorithms, we hand tune the parameters for their best performance, and same step sizes are used when applicable. Note that SPSA must use a very small step size ($\alpha=0.01$) in max-$s$-squared-sum problem, or it will diverge.

\paragraph{Re-shuffling the blocks.} Note that the max-$s$-squared-sum function does not satisfy the Lipschitz differentiability condition (\textit{i.e.} Assumption~\ref{assumption:Lipschitz_Diff}). Moreover the gradient support changes, making this an extremely difficult zeroth-order problem. To overcome the difficulty of discontinuous gradients, we apply an additional step that re-shuffles the blocks every $J$ iterations. This re-shuffling trick is not required for the problems that satisfies our assumptions; nevertheless, we observe very similar convergence behavior with slightly more queries when the re-shuffling step was applied on the problems that satisfy the aforementioned assumptions.

\subsection{Settings for sparse DWT attacks on images}
We allows a total query budget of $10,000$ for all tested algorithms in each image attack, i.e. an attack is considered a failure if it cannot change the label within $10,000$ queries. We use a 3-level `db45'  wavelet transform. All the results present in Section~\ref{sec:ImageAttack} use the half-point symmetric boundary extension, hence the wavelet domain has a dimension of $676,353$; slightly larger than the original pixel domain dimension. For the interested reader, a discussion and more results about other boundary extensions can be found in Appendix~\ref{app:more_experimental_result}. 
 
For all variations of ZO-BCD, we choose the block size to be $170$ with per block sparsity $s_\mathrm{block}=10$, thus $m=52$ queries are used per iteration. Sampling radius is set to be $10^{-2}$. The block gradient estimation step runs at most $n=30$ CoSaMP iterations, and step size $\alpha=10$.


\subsection{Settings for sparse CWT attacks on audio signals}
For both targeted and untargeted CWT attack, we use Morse wavelets with $111$ frequencies, which significantly enlarges the problem dimension from $16,000$ to $1,776,000$ per clip. In the attack, we choose the block size to be $295$ with per block sparsity $s_\mathrm{block}=9$, thus $m=52$ queries are used per iteration. The sampling radius is $\delta = 10^{-3}$. The block gradient estimation step runs at most $n=30$ CoSaMP iterations. In targeted attacks, the step size is $\alpha=0.05$ , and Table~\ref{tab:untargeted_audio} specifies the step size used in untargeted attacks.

In Table~\ref{tab:untargeted_audio}, we also include a result of voice domain attack for the comparison. The parameter settings are same with aforementioned untargeted CWT attack settings, but we have to reduce step size to $0.01$ for stability. Also, note that the problem domain is much smaller in the original voice domain, so the number of blocks is much less while we keep the same block size.



\section{More experimental results and discussion} \label{app:more_experimental_result}
\subsection{Sparse DWT attacks on images}
\paragraph{Periodic extension.} 
As mentioned, when we use boundary extensions other than periodic extension, the dimension of the wavelet coefficients will increase, depending on the size of wavelet filters and the level of the wavelet transform. More precisely, wavelets with larger support and/or deeper levels of DWT result in a larger increase in dimensionality. On the other hand, using periodic extension will not increase the dimension of wavelet domain. We provide test results for both boundary extensions in this section for the interested reader. 

\paragraph{Compressed attack.} As discussed in Section~\ref{sec:sparse_attack}, DWT is widely used in data compression, such as JPEG-2000 for image. In reality, the compressed data are often saved in the form of sparse (and larger) DWT coefficients after vanishing the smaller ones. While we have already tested an attack on the larger DWT coefficients only (see Section~\ref{sec:ImageAttack}), it is also of interest to test an attack after compression. That is, we zero out the smaller wavelet coefficients ($\mathrm{abs}\leq 0.05$) first, and then attack on only the remaining, larger coefficients.

We use the aforementioned parameter settings for these two new attacks. We present the results in Table~\ref{tab:attack_result_full}, and for completeness we also include the results already presented in Table~\ref{tab:attack_result}. One can see that ZO-BCD-R(compressed) has higher attack success rate and lower $\ell_2$ distortion, exceeding the prior state-of-the-art as presented in Table~\ref{tab:attack_result}. We caution that this is not exactly a fair comparison with prior works however, as the compression step modifies the image before the attack begins. 

\begin{table}[t]
\caption{Results of untargeted image adversarial attack using various algorithms. Attack success rate (ASR), average final $\ell_2$ distortion, and number of queries till first successful attack. ZO-BCD-R(periodic ext.) stands for ZO-BCD-R applying periodic extension for implementing the wavelet transform. ZO-BCD-R(compressed) stands for applying ZO-BCD-R to attack only large wavelet coefficients (abs $\geq0.05$) and vanishing the smaller values. The other methods are the same in Table~\ref{tab:attack_result}.} \label{tab:attack_result_full}
\vspace{0.1in}
\begin{center}
\begin{small}
\begin{tabular}{lccc}
\toprule
\textsc{Methods} & \textsc{ASR} & \textsc{$\ell_2$ dist}  & \textsc{Queries}\\
\midrule
ZO-SCD    & 78$\%$  & 57.5 & 2400\\
ZO-SGD  & 78$\%$ &  37.9 & 1590\\
ZO-AdaMM  & 81$\%$ &  28.2 & 1720\\
ZORO      & 90$\%$  &  21.1 &   2950 \\
ZO-BCD-R & 92$\%$  &  14.1 &   2131 \\
ZO-BCD-RC & 92$\%$  &  14.2 &  2090 \\
ZO-BCD-R(periodic ext.)  & 95$\%$ &  21.0 &  1677 \\
ZO-BCD-R(compressed) & $\textbf{96}\%$  &  $\textbf{13.1}$ &   $\textbf{1546}$ \\
ZO-BCD-R(large coeff.) & $\textbf{96}\%$&  13.7 &   1662 \\
\bottomrule
\end{tabular}
\end{small}
\end{center}
\vspace{-0.05 in}
\end{table}

\paragraph{Defense tests.} Finally, we also tested some simple mechanisms for defending against our attacks; specifically harmonic denoising. We test DWT with the famous Haar wavelets, DWT with db45 wavelets which is also used for attack, and the essential discrete cosine transform (DCT). The defense mechanism is to apply the transform to the attacked images and then denoise by zeroing out small wavelet coefficients before transforming back to the pixel domain. We only test the defense on images that were successfully attacked. Tables~\ref{tab:defense_result_regular} and \ref{tab:defense_result_largecoeff} show the results of defending against the ZO-BCD-R and ZO-BCD-R(large coeff.) attacks respectively. Interestingly, using the attack wavelet (\textit{i.e.} db45) in defence is not a good strategy. We obtain the best defense results by using a mismatched transform ({\em i.e.} DCT or Haar defense for a db45 attack) and a small thresholding value.

\begin{table}[t]
\caption{Defense of image adversarial wavelet attack by ZO-BCD-R. Defense recovery success rate under haar and db45 wavelet filters, and discrete cosine transform (DCT) filter. Thresholding values of 0.05, 0.10, 0.15, and 0.20 were considered.} \label{tab:defense_result_regular}
\vspace{0.1in}
\begin{center}
\begin{small}
\begin{tabular}{lrrrr}
\toprule
\textsc{Defence Methods} & \textsc{0.05} & \textsc{0.10} & \textsc{0.15}  & \textsc{0.20}\\
\midrule
Haar    & 74$\%$ &  $75\%$ & $76\%$ & $75\%$\\
db45  & 74$\%$&  $72\%$ & $71\%$ & $63\%$\\
DCT  & 72$\%$&  $\textbf{79}\%$ & $75\%$ & $67\%$\\
\bottomrule
\end{tabular}
\end{small}
\end{center}
\vspace{-0.05 in}
\end{table}

\begin{table}[t]
\caption{Defense of image adversarial wavelet attack by ZO-BCD-R (large coeff.). Defense recovery success rate under Haar and db45 wavelet filters, and discrete cosine transform (DCT) filter. Thresholding values 0.05, 0.10, 0.15, and 0.20 were considered.} \label{tab:defense_result_largecoeff}
\vspace{0.1in}
\begin{center}
\begin{small}
\begin{tabular}{lrrrr}
\toprule
\textsc{Defence Methods} & \textsc{0.05} & \textsc{0.10} & \textsc{0.15}  & \textsc{0.20}\\
\midrule
Haar    & 75$\%$ &  $72\%$ & $\textbf{78}\%$ & $76\%$\\
db45  & 71$\%$&  $70\%$ & $69\%$ & $63\%$\\
DCT  & 62$\%$&  $74\%$ & $68\%$ & $58\%$\\
\bottomrule
\end{tabular}
\end{small}
\end{center}
\vspace{-0.05 in}
\end{table}

\paragraph{More adversarial examples.} In Figure~\ref{fig:AttackedImages}, we presented some adversarial images generated by the ZO-BCD-R and ZO-BCD-R(large coeff.) attacks. For the interested reader, we present more visual examples in Figure~\ref{fig:MoreAttackedImages}, and include the adversarial attack results generated by all versions of ZO-BCD.

\begin{figure*}[t]
\centering 
\subfloat[\scriptsize ZO-BCD-R: ``frying pan'' $\rightarrow$ ``strainer'']{\includegraphics[width=.19\linewidth]{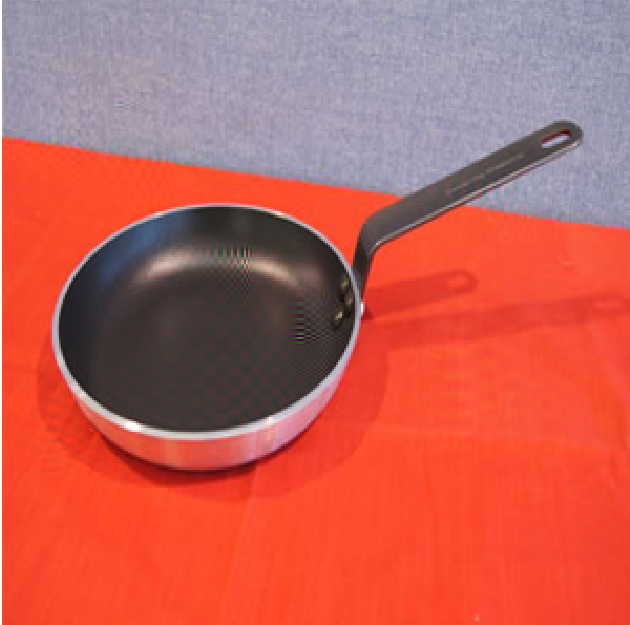}}
\hfill
\subfloat[\scriptsize ZO-BCD-RC: ``frying pan'' $\rightarrow$ ``strainer'']{\includegraphics[width=.19\linewidth]{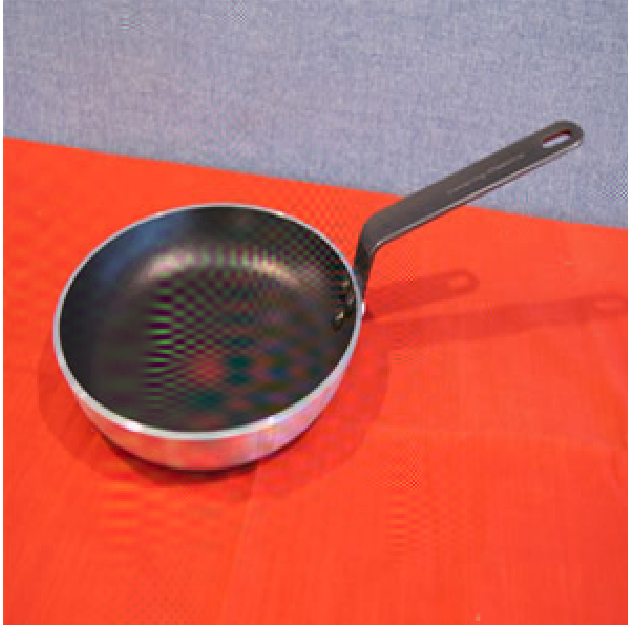}}
\hfill
\subfloat[\scriptsize ZO-BCD-R (periodic ext.): ``frying pan'' $\rightarrow$ ``strainer'']{\includegraphics[width=.19\linewidth]{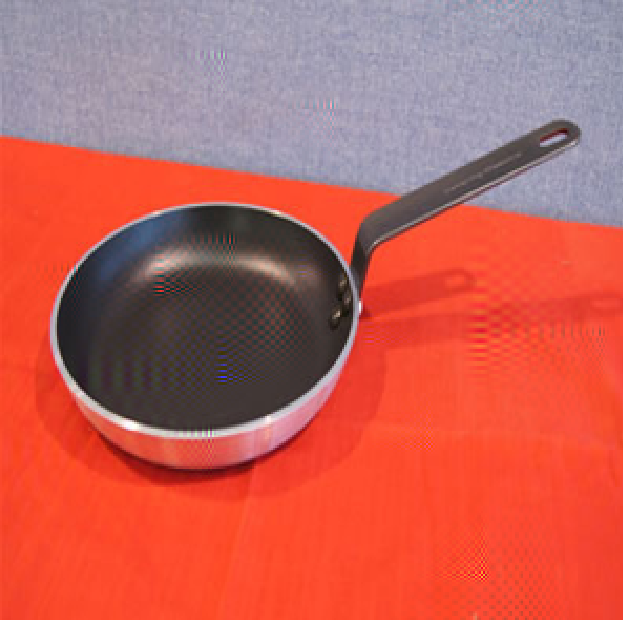}}
\hfill
\subfloat[\scriptsize ZO-BCD-R (compressed): ``frying pan'' $\rightarrow$ ``strainer'']{\includegraphics[width=.19\linewidth]{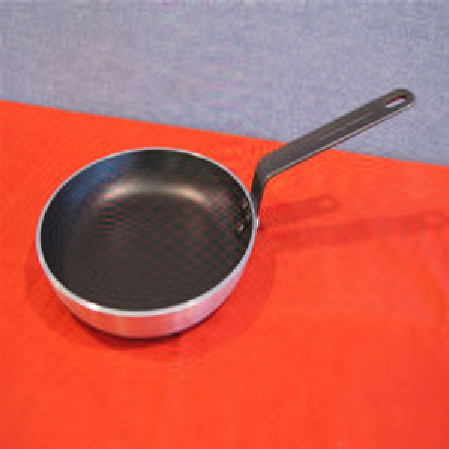}}
\hfill
\subfloat[\scriptsize ZO-BCD-R (large coeff.): ``frying pan'' $\rightarrow$ ``strainer'']{\includegraphics[width=.19\linewidth]{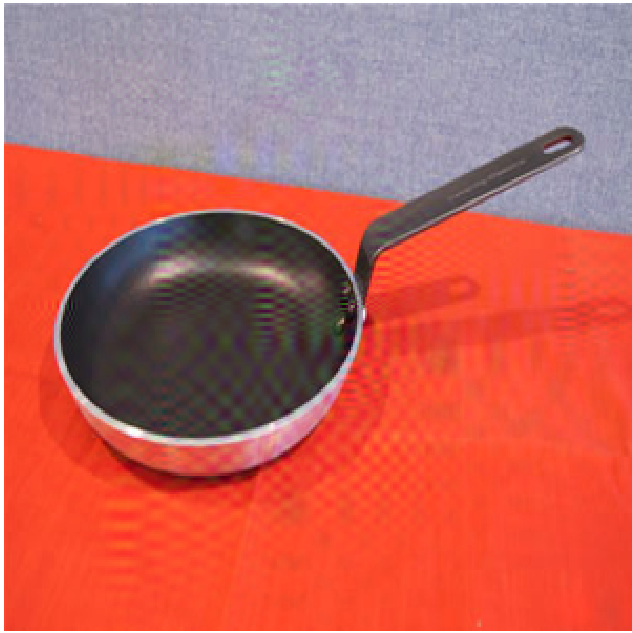}}
\\
\subfloat[\scriptsize ZO-BCD-R: ``strawberry'' $\rightarrow$ ``pomegranate'']{\includegraphics[width=.19\linewidth]{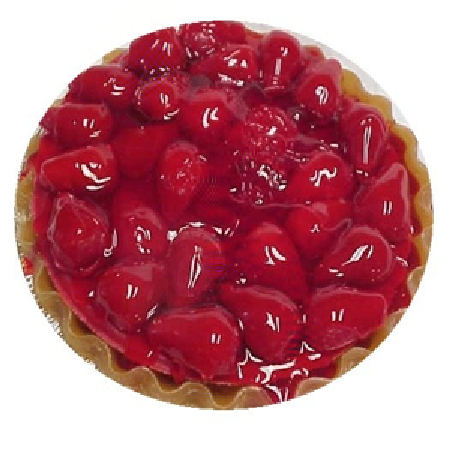}}
\hfill
\subfloat[\scriptsize ZO-BCD-RC: ``strawberry'' $\rightarrow$ ``pomegranate'']{\includegraphics[width=.19\linewidth]{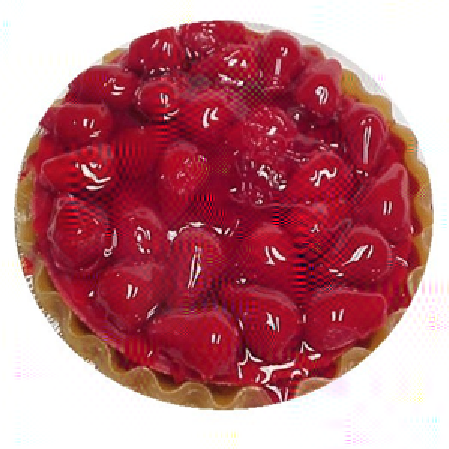}}
\hfill
\subfloat[\scriptsize ZO-BCD-R (periodic ext.): ``strawberry'' $\rightarrow$ ``pomegranate'']{\includegraphics[width=.19\linewidth]{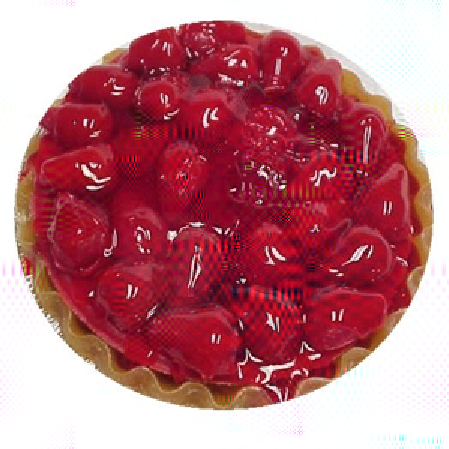}}
\hfill
\subfloat[\scriptsize ZO-BCD-R (compressed): ``strawberry'' $\rightarrow$ ``pomegranate'']{\includegraphics[width=.19\linewidth]{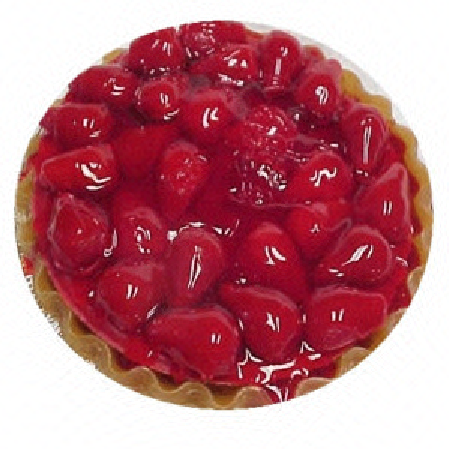}}
\hfill
\subfloat[\scriptsize ZO-BCD-R (large coeff.): ``strawberry'' $\rightarrow$ ``pomegranate'']{\includegraphics[width=.19\linewidth]{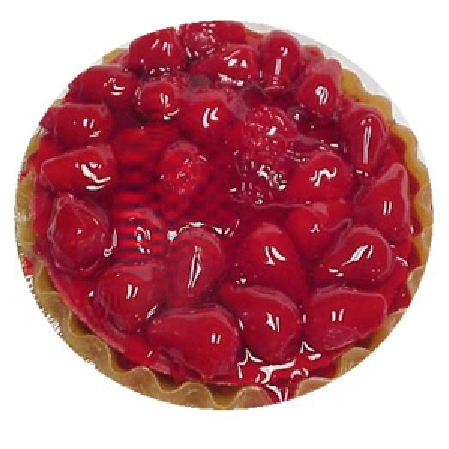}}
\caption{More examples of wavelet attacked images by ZO-BCD-R, ZO-BCD-RC, ZO-BCD-R(periodic ext.), ZO-BCD-R(compressed), and ZO-BCD-R(large coeff.), with true labels and mis-classified labels.}
\label{fig:MoreAttackedImages}
\end{figure*}

\subsection{Sparse CWT attacks on audio signals}
\label{appendix:Audio}

Adversarial attacks on speech recognition is a more nebulous concept than that of adversarial attacks on image classifiers, with researchers considering a wide variety of threat models. In \citep{cisse2017houdini}, an attack on the speech-to-text engine {\tt DeepSpeech} \citep{hannun2014deep} is successfully conducted, although the proposed algorithm, Houdini, is only able to force minor mis-transcriptions (``A great saint, saint Francis Xavier'' becomes ``a green thanked saint fredstus savia''). In \citep{carlini2018audio}, this problem is revisited, and they are able to achieve 100\% success in targeted attacks, with any initial and target transcription from the Mozilla Common Voices dataset (for example, ``it was the best of times, it was the worst of times'' becomes ``it is a truth universally acknowledged that a single''). We emphasize that both of these attacks are {\em whitebox}, meaning that they require full knowledge of the victim model. We also note that speech-to-text transcription is not a classification problem, thus the classic Carlini-Wagner loss function so frequently used in generating adversarial examples for image classifiers cannot be straightforwardly applied. The difficulty of designing an appropriate attack loss function is discussed at length in \citep{carlini2018audio}.

A line of research more related to the current work is that of attacking keyword classifiers. Here, the victim model is a classifier; designed to take as input a short audio clip and to output probabilities of this clip containing one of a small, predetermined list of keywords (``stop'', ``left'' and so on). Most such works consider the {\tt SpeechCommands} dataset \citep{warden2018speech}. To the best of the author's knowledge, the first paper to consider targeted attacks on a keyword classification model was \citep{alzantot2018did}, and they do so in a black-box setting. They achieve an $87\%$ attack success rate (ASR) using a genetic algorithm, whose query complexity is unclear. They do not report the relative loudness of their attacks; instead they report the results of a human study in which they asked volunteers to listen to and label the attacked audio clips. They report that for $89\%$ of the successfully attacked clips human volunteers still assigned the clips the correct ({\em i.e.} source) label, indicating that these clips were not corrupted beyond comprehension. Their attacks are per-clip, {\em i.e.} not universal. 

In \citep{vadillo2019universal}, universal and untargeted attacks on a {\tt SpeechCommands} classifier are considered. Specifically, they seek to construct a distortion $\delta$ such that for {\em any } clip $x$ from a specified source class ({\em e.g.} ``left''), the attacked clip $x + \delta$ is misclassified to a source class ({\em e.g.} ``yes'') by the model. They consider several variations on this theme; allowing for multiple source classes. The results we recorded in Section~\ref{sec:AudioAttack} (ASR of $70.3\%$ at a remarkably low mean relative loudness of $-41.63$ dB) were the best reported in the paper, and were for the single-source-class setting. This attack was in the white-box setting. 

Finally, we mention two recent works which consider very interesting, but different threat models. \citep{li2020advpulse} considers the situation where a malicious attacker wishes to craft a short (say $0.5$ second long) that can be added {\em to any part} of a clean audio clip to force a misclassification. Their attacks are targeted and universal, and conducted in the white-box setting. They do not report the relative loudness of their attacks. In \citep{xie2020enabling}, a generative model is trained that takes as input a benign audio clip $x$, and returns an attacked clip $x + \delta$. The primary advantage of this approach is that attacks can be constructed on the fly. In the targeted, per-clip white-box setting they achieve the success rate of $93.6\%$ advertised in Section~\ref{sec:AudioAttack}, at an approximate relative loudness of $-30$ dB. They also consider universal attacks, and a transfer attack whereby the generative model is trained on a surrogate classification model. 

In all the aforementioned keyword attacks, the victim model is some variant of the model proposed in \citep{sainath2015convolutional}. Specifically, the audio input is first transformed into a 2D spectrogram using Mel frequency coefficients, bark coefficients or similar. Then, a 3- to 5- layer convolutional neural network is applied.

\end{document}